\documentclass[twoside,11pt]{article}

\usepackage{blindtext}

\usepackage[preprint]{jmlr2e}

\usepackage{color}
\usepackage[ruled,vlined]{algorithm2e}

\usepackage{caption}
\usepackage{subcaption}
\usepackage{booktabs}
\usepackage{enumitem}

\usepackage{physics}
\usepackage[para]{footmisc}

\newtheorem{assumption}{Assumption}[section]

\definecolor{orange}{rgb}{1, 0.5, 0}

\definecolor{darkspringgreen}{rgb}{0.09, 0.45, 0.27}

\renewcommand{\chi}{\mathrm{I}}

\DeclareMathOperator*{\argmin}{argmin}
\DeclareMathOperator*{\argmax}{argmax}

\newcommand{\R}{\mathbb R}

\newcommand{\TR}{\mathrm{tr}}

\def\eps{\varepsilon}

\newcommand{\D}{\mathbb D}
\newcommand{\Sym}{\mathrm{Sym}}
\def\diag{{\rm diag}}
\def\diagvec{{\rm diagvec}}

\usepackage{lastpage}
\jmlrheading{23}{2025}{1-\pageref{LastPage}}{1/21; Revised 5/22}{9/22}{21-0000}{De Marinis, Guglielmi, Sicilia and Tudisco}

\ShortHeadings{Improving the robustness of neural ODEs with minimal perturbation}{De Marinis, Guglielmi, Sicilia and Tudisco}
\firstpageno{1}

\begin{document}

\title{Improving the robustness of neural ODEs with minimal weight perturbation}


\author{\name Arturo De Marinis \email arturo.demarinis@gssi.it \\
    \addr Division of Mathematics \\
    Gran Sasso Science Institute \\ L'Aquila, Italy
    \AND
    \name Nicola Guglielmi \email nicola.guglielmi@gssi.it \\
    \addr Division of Mathematics\\ Gran Sasso Science Institute\\ L'Aquila, Italy
    \AND
    \name Stefano Sicilia \email stefano.sicilia@umons.ac.be \\
    \addr Department of Mathematics and Operational Research\\ University of Mons\\ Mons, Belgium
    \AND
    \name Francesco Tudisco \email f.tudisco@ed.ac.uk \\
    \addr Maxwell Institute \& School of Mathematics\\ The University of Edinburgh\\ Edinburgh, United Kingdom
}

\editor{My editor}

\maketitle

\begin{abstract}
 We propose a method to enhance the stability of a neural ordinary differential equation (neural ODE) by reducing the maximum error growth subsequent to a perturbation of the initial value. Since the stability depends on the logarithmic norm of the Jacobian matrix associated with the neural ODE, we control the logarithmic norm by perturbing the weight matrices of the neural ODE by a smallest possible perturbation (in Frobenius norm). We do so by engaging an eigenvalue optimisation problem, for which we propose a nested two-level algorithm. For a given perturbation size of the weight matrix, the inner level computes optimal perturbations of that size, while - at the outer level - we tune the perturbation amplitude until we reach the desired uniform stability bound. We embed the proposed algorithm in the training of the neural ODE to improve its robustness to perturbations of the initial value, as adversarial attacks. Numerical experiments on classical image datasets show that an image classifier including a neural ODE in its architecture trained according to our strategy is more stable than the same classifier trained in the classical way, and therefore, it is more robust and less vulnerable to adversarial attacks.
\end{abstract}

\begin{keywords}
Neural ODEs, adversarial robustness, stability, deep learning, dynamical systems
\end{keywords}

\section{Introduction}

Neural ordinary differential equations (neural ODEs) \citep{haber2017stable,chen2018neural} model the forward pass of a neural network as the evolution of a feature vector \( x(t) \) through a continuous-time dynamical system of the form
\begin{equation}\label{eq:neuralODE}
    \dot{x}(t) := \dv{}{t}x(t) = f(x(t), t, \theta), \quad t\in[0,T],
\end{equation}               
where \( x(t) \in \mathbb{R}^n \) represents the state, and \( f \) is a neural network parameterised by \( \theta \). This formulation generalises residual and recurrent neural networks and has been extensively explored in modern machine learning, achieving notable success across diverse applications with various choices of \( f \), including second-order damped oscillators \citep{rusch2021coupled}, state-space models \citep{gu2020hippo,gu2023mamba}, diffusion-based generative models \citep{song2021score,karras2022elucidating}, and graph neural networks \citep{chamberlain2021grand,nguyen2024coupled}.

In this article we shall mainly consider one-layer weight-tied neural ODEs of the form (\cite{chen2018neural})
\begin{equation*}\label{eq:neuralODE1}
    \dot{x}(t) = \sigma (A x(t) + b), \quad t\in[0,T],
\end{equation*}
where $\sigma:\R\to\R$ is a smooth activation function applied entry-wise, $T>0$ is the time horizon, $A\in\R^{n\times n}$ is the weight matrix and $b\in\R^n$ is the bias vector. A composition of functions $\sigma (A_i x + b_i)$, $i=1,2,\ldots$, allows to obtain a multi-layer neural ODE.

A neural ODE of the form \eqref{eq:neuralODE} naturally leads to discrete architectures through numerical integration. For instance, applying the Euler method to \eqref{eq:neuralODE} over a partition \mbox{$0=t_0< t_1 < \dots < t_N =T$} of $[0,T]$ with step size $ h = T/N $ results in a residual neural network \citep{ResNet}
\[
    x_{k+1} = x_k + h f(x_k, t_k, \theta), \quad k=0,1,\dots,N-1,
\]
where $ x_k \approx x(t_k) $, and the number of discretisation points $N$ determines the depth of the neural network.

One of the key advantages of formulating neural networks via ODEs is the ability to study their theoretical properties through dynamical systems analysis, particularly in terms of stability, contractivity and conservation laws \citep{celledoni2021structure}. These properties are crucial for understanding the robustness of deep learning models, particularly against perturbations in the input data.

In this work, we focus on designing deep neural networks as discretisations of neural ODEs that are inherently robust against data perturbations, including adversarial attacks, carefully crafted perturbations in input designed to mislead the network into making incorrect predictions \citep{biggio2013evasion,szegedy2013intriguing,goodfellow2014explaining}. To analyse robustness, we consider the sensitivity of the solution of \eqref{eq:neuralODE} to variations in the initial condition. For a given norm $\|\cdot\|$ on $\R^n$, under suitable Lipschitz assumptions on the vector field $f$, a neural ODE satisfies the bound
\begin{equation}\label{eq:stab_cond}
    \|x_1(T)-x_2(T)\| \leq C \|x_1(0)-x_2(0)\|,
\end{equation}
for a constant \( C > 0 \),  and any two solutions \( x_1(t) \) and \( x_2(t) \) corresponding to initial conditions \( x_1(0) \neq x_2(0) \). This constant \( C \) governs the stability behaviour of the system:
\begin{itemize}
    \item if \( C \gg 1 \), small perturbations in input may amplify significantly, leading to instability;
    \item if \( C \) is moderate, the neural ODE is stable;
    \item if \( C < 1 \), the neural ODE is contractive, meaning that perturbations decay over time.
\end{itemize}

From a numerical ODE point of view, a key challenge is ensuring that numerical discretisations of \eqref{eq:neuralODE} preserve these stability properties \citep{dahlquist1979generalized}. Assuming this is the case, by designing neural ODE models that maintain stability or contractivity, we can enhance robustness against input perturbations, including adversarial ones. This can be achieved through explicit regularisation in the loss function \citep{yan2020robustness, kang2021stable, li2022defending} or by imposing specific structural constraints on the network parameters \citep{haber2017stable,celledoni2021structure,guglielmi2024contractivity}. 

However, enforcing strong stability or contractivity can degrade model accuracy. In classification tasks, requiring contractivity may significantly reduce the fraction of correctly classified test samples. This trade-off is inevitable: overly accurate models tend to be unstable \citep{gottschling2020troublesome}, while strongly stable models may lack expressiveness. Consequently, an optimal balance must be found, where \( C > 1 \) remains moderate, ensuring both accuracy and robustness.

Building on this insight, we propose an optimisation strategy that optimises the network parameters to balance accuracy and stability. Specifically, we formulate an optimisation problem where we seek the closest (structured) perturbation of the weight matrices that enforces a desired stability bound for \( C \). This approach allows us to explicitly control the Lipschitz constant while maintaining high predictive performance. 

We validate our method through a range of numerical experiments against existing baseline methods to stabilise neural ODEs on standard classification benchmarks, demonstrating that our approach effectively improves robustness while preserving competitive accuracy, outperforming existing approaches.

\subsection{Paper organisation}

The paper is organised as follows. In Section \ref{sec:stab} the stability of neural ODEs is studied from a theoretical point of view; contractivity in the $2$-norm is recalled and adapted to the considered problem. Next, a two-level algorithm is proposed, with the aim of modifying the matrix $A$ in order to reach the stability goal for the solution of the neural ODE. The two-level approach, taken into account here, uses an {\it inner iteration} to compute the solution of a suitable eigenvalue optimisation problem associated to the required stability property (which requires the uniform control of the logarithmic norm of the matrix $A$ premultiplied by a diagonal matrix), where the allowed perturbations of $A$ have a fixed perturbation size $\eps$. Then the optimal perturbation size $\eps_\star$ is determined in an {\it outer iteration}, that is the perturbation of minimal size that allows to obtain the desired stability bound.   
Section \ref{sec:inner} describes the inner iteration of the two-level algorithm, and Section \ref{sec:outer} discusses the solution of the outer iteration. Section \ref{sec:num_exp} shows how the proposed algorithm is embedded in the training of the neural ODE to improve its stability and robustness to adversarial attacks. Numerical experiments on MNIST, FashionMNIST, CIFAR10 and SVHN datasets classification are reported. Finally, Section \ref{sec:conc} draws our conclusions.

\subsection{Existing literature on the stability of neural ODEs}

Many papers have explored the idea of interpreting and studying deep residual neural networks as discretisations of continuous-time ordinary differential equations (ODEs). Early works in this direction include \citep{haber2017stable, weinan2017proposal, chen2018neural, lu2018beyond, ruthotto2020deep}, with \cite{chen2018neural} popularising the term \textit{Neural Ordinary Differential Equation}. As noted already in \citep{haber2017stable}, tools from the stability theory of ODEs can be highly effective for analysing and enhancing the adversarial robustness of residual networks modelled as discretisations of continuous-time neural ODEs. Since then, the literature on this topic has expanded significantly, with many recent works focusing on the stability properties of neural ODEs.

\cite{carrara2019robustness} analyse the robustness of image classifiers implemented as neural ODEs against adversarial attacks and compare their performance to standard deep models. They empirically demonstrate that neural ODEs are inherently more robust to adversarial attacks than state-of-the-art residual networks.

\cite{yan2020robustness} reach similar conclusions. They offer a deeper understanding of this phenomenon by leveraging a desirable property of the flow of a continuous-time ODE: integral curves are non-intersecting. For three initial points, the integral curve starting from the middle point is always sandwiched between the integral curves starting from the others. This imposes a bound on the deviation from the original output caused by a perturbation in input. In contrast, state-of-the-art residual networks lack such bounds on output deviation, making neural ODEs intrinsically more robust.

\cite{li2020implicit} investigate the vulnerability of neural networks to adversarial attacks from the perspective of dynamical systems. By interpreting a residual network as an explicit Euler discretisation of an ordinary differential equation, they observe that the adversarial robustness of a residual network is linked to the numerical stability of the corresponding dynamical system. Specifically, more stable numerical schemes may lead to more robust deep networks. Since implicit methods are more stable than explicit ones, they propose using the implicit Euler method instead of the explicit one to discretise the neural ODE, resulting in a more robust model.

\cite{carrara2021defending, carrara2022improving} and \cite{caldelli2022tuning} demonstrate that adversarial robustness can be enhanced by operating the network in different tolerance regimes during training and testing. Notably, robustness is further improved when the test-time tolerance is greater than the training-time tolerance, which is typically the same value used by the adversarial attacker.

\cite{kang2021stable} and \cite{li2022defending} propose a Stable neural ODE with Lyapunov-stable equilibrium points for Defending against adversarial attacks (SODEF). By ensuring that the equilibrium points of the ODE are Lyapunov-stable, and by minimising the maximum cosine similarity between them, the solution for an input with a small perturbation converges to the same solution as the unperturbed initial input.

\cite{huang2022adversarial} introduce a provably stable architecture for neural ODEs that achieves non-trivial adversarial robustness even when trained without adversarial examples. Drawing inspiration from dynamical systems theory, they design a stabilised neural network whose neural ODE blocks are not symmetric, proving it to be input-output stable.

\cite{cui2023robustness}, motivated by a family of weight functions used to enhance the stability of dynamical systems, propose a novel activation function called half-Swish. Numerical experiments show that their model can achieve superior performance on robustness both against stochastic noise images and
adversarial attacks, while maintaining competitive results of classification accuracy over the clean testing data.

\cite{purohit2023ortho} proposes to use orthogonal convolutional layers to define the neural network dynamics of the neural ODE. This allows to upper bound the Lipschitz constant of the dynamics, making the model robust. Indeed, the representations of clean samples and their adversarial counterparts remain close.

\cite{haber2017stable}, \cite{ruthotto2020deep}, \cite{celledoni2023dynamical} and \cite{sherry2024designing} introduce neural ODEs whose architectures and parameters are structured. The structure is designed in such a way that the deep neural networks, resulting from their numerical discretisations, are inherently non-expansive, making them robust to adversarial attacks by construction.

\section{Stability of neural ODEs} \label{sec:stab}

As in \citep{kang2021stable} and  \citep{li2022defending}, our goal is to make a neural ODE stable. In order to explain our methodology, we start by considering a one-layer weight-tied neural ODE  
\begin{equation}\label{eq:1lnODE}
    \dot{x}(t) = \sigma(Ax(t)+b), \quad t\in[0,T],
\end{equation}
where $x(t)\in\R^n$ is the feature vector evolution function, $A\in\R^{n\times n}$ and $b\in\R^n$ are the parameters, and $\sigma:\R\to\R$ is a smooth activation function applied entry-wise. Later, we will provide details on how to extend to the two-layer weight-tied neural ODE
\[
    \dot{x}(t) = \sigma(A_2\sigma(A_1x(t)+b_1)+b_2), \quad t\in[0,T],
\]
which then directly transfer to the general multi-layer case.

We require that $\sigma'(\R)\subseteq[m,1]$, where $0<m\le1$ is a parameter. For most activation functions, $\sigma'(\R)\subseteq[0,1]$, but - to our aims - we assume $\sigma'(x) \ge m > 0$, for all $x\in\R$. This is easily achievable by means of a slight modification of the activation function without affecting the approximation capabilities of the neural ODE.

For any pair of solutions $x_1(t)$ and $x_2(t)$, corresponding to different initial data $x_1(0)$ and $x_2(0)$, our goal is to obtain an a priori bound of the form \eqref{eq:stab_cond}, i.e. 
\[
    \|x_1(t)-x_2(t) \|  \le C \|x_1(0)-x_2(0) \| , \quad t\in[0,T],
\]
for a moderately sized constant $C>0$. By denoting by $\|\cdot\|$ and $\langle \cdot , \cdot \rangle$ the $2$-norm and its associated standard inner product, and $f(t,x)=\sigma(Ax+b)$, the stability condition above is satisfied if the one-sided Lipschitz condition holds   
\begin{equation}\label{eq:1sLcond}
    \langle f(t,x) - f(t,y) , x - y \rangle \le \delta_\star \| x - y  \| ^2, \quad t\in[0,T],\ x,y\in\R^n,
\end{equation}
for a certain  $\delta_\star=\delta_\star(A)\in\R$ , which implies $C=e^{\delta_\star T}$.
 
In our setting, we characterise the constants $\delta_\star$ and $C$ by means of the logarithmic norm of the Jacobian of the vector field $f(t,x)$. In particular we consider the logarithmic 2-norm of a matrix $M\in \R^{n\times n}$, that is defined as
\begin{equation}\label{eq:l2n}
    \mu_2(M) = \lambda_{\max}(\Sym(M)) := \max_{i=1,\dots,n} \lambda_i(\Sym(M)),
\end{equation}
where $\lambda_i(\Sym(M))$ denotes the $i$-th eigenvalue of the symmetric part of $M$.

Then, the one-sided Lipschitz condition \eqref{eq:1sLcond} becomes
\begin{align*}
    &\langle \sigma(Ax+b) - \sigma(Ay+b) , x - y \rangle = \langle J (x-y) , x-y \rangle = (x-y)^\top J(x-y) \\
    &= (x-y)^\top \Sym(J) (x-y),
\end{align*}
where $J$ is the Jacobian matrix of  $\sigma(A\cdot +b)$ evaluated at $\xi_\star$, 
\[
    J=\dv{}{x} \sigma(Ax+b)\bigg|_{x=\xi_\star}\in\R^{n\times n},
\]
where $\xi_\star=(1-t_\star)x+t_\star y$ for some $t_\star\in[0,1]$, as provided by the mean value theorem. We note that, defined $\D^{n\times n}$ as the set of diagonal matrices of size $n$ and
\[
 \Omega_m = \{ D\in \D^{n\times n}\ :\ m \le D_{ii} \le 1,\ \forall i = 1,\dots,n\},
\]
the Jacobian $J$ takes the form
\[
    J = \diag(d)A:= DA,
\]
where $d=\sigma'(A\xi_\star+b)$ and - for a vector $v \in \R^n$  - $\diag(v)$ is the diagonal matrix whose diagonal is the vector $v$. 
Furthermore, we note that writing $x-y$ in terms of $u_1,\dots,u_n\in\R^n$, the orthonormal eigenvectors of $\Sym(J)$,
\[
    x-y=\sum_{i=1}^n\alpha_i u_i, \qquad \alpha_i\in\R,
\]
yields
\begin{align*}
    &(x-y)^\top \Sym(J) (x-y) = \sum_{i=1}^n \alpha_i u_i^\top \sum_{j=1}^n \alpha_j \Sym(J) u_j = \sum_{i=1}^n \lambda_i(\Sym(J)) \alpha_i^2 \\
    &\le \left( \max_{i=1,\dots,n} \lambda_i(\Sym(J)) \right) \sum_{i=1}^n \alpha_i^2 = \mu_2(J) \|x-y \| ^2.
\end{align*}
By combining all together, we get
\begin{equation*}
    \langle \sigma(Ax+b) - \sigma(Ay+b) , x - y \rangle \le \mu_2(J) \|x-y \| ^2 \\
    \le \max_{D\in\Omega_m} \mu_2(DA) \|x-y \| ^2 := \delta_\star \|x-y \| ^2,
\end{equation*}
i.e. the neural ODE \eqref{eq:1lnODE} satisfies the one-sided Lipschitz condition \eqref{eq:1sLcond} with
\begin{equation}\label{eq:nu}
    \delta_\star := \max_{D\in\Omega_m} \mu_2(DA)
\end{equation}
and thus the stability condition \eqref{eq:stab_cond} with $C=e^{\delta_\star T}$. According to the sign of $\delta_\star$, we get stability ($\delta_\star>0$ small), nonexpansivity ($\delta_\star=0$) and contractivity ($\delta_\star<0$).

In light of \eqref{eq:nu} and the final remarks following \eqref{eq:stab_cond}, we would like the weight matrix $A$ to be such that $\delta_\star$ is a positive constant of moderate size, so that the neural ODE \eqref{eq:1lnODE} does not significantly amplify the perturbation in input. Since this fact generally does not happen if the neural ODE is trained using standard gradient methods, we fix $\delta\in\R$, $\delta<\delta_\star$, we compute the closest matrix $B$ to $A$ in a certain metric such that $\mu_2(DB)=\delta$, for all $D\in\Omega_m$, and we replace $A$ by $B$ in the neural ODE. More precisely, we compute a matrix of the type
\begin{equation} \label{eq:proj}
    A_{\delta} \in A + \argmin_{\Delta\in\R^{n\times n}} \left\{ \| \Delta \|_F \ : \ \max_{D\in\Omega_m} \mu_2(D(A+\Delta)) = \delta \right\},
\end{equation}
where the Frobenius inner product of two matrices $M,N\in\R^{n\times n}$ and the corresponding Frobenius norm are defined as
\[
    \langle M , N \rangle_F = \TR(N^\top M) = \sum_{i,j=1}^{n} M_{ij}N_{ij}, \qquad \|M\|_F = \sqrt{ \langle M , M \rangle_F }.
\]
In this framework, it could be desirable that the perturbation $\Delta$ has a certain structure, i.e. $\Delta\in\mathcal{S}$, where $\mathcal{S}$ is a subspace of $\R^{n\times n}$, for instance the space of diagonal matrices. In this case, we can redefine the problem as
\begin{equation} \label{eq:str_proj}
    A_{\delta}^\mathcal{S} \in A + \argmin_{\Delta\in\mathcal{S}} \left\{ \| \Delta \|_F \ : \ \max_{D\in\Omega_m} \mu_2(D(A+\Delta)) = \delta \right\}.
\end{equation}
The definitions \eqref{eq:proj} and \eqref{eq:str_proj} involve the solution of a matrix nearness problem that is the main effort of this paper (as a general reference to address this problem by continuous optimisation tools we refer the reader to \cite{guglielmi2025matrix}). We notice that such matrix nearness problem is not convex  and the matrices  $A_{\delta}$ and $A_{\delta}^\mathcal{S}$ may not be unique since there could be multiple minimisers. However, in the numerical experiments in Section \ref{sec:num_exp} we will consider the solution computed by the algorithm. Moreover, the choice of $\delta$ as a positive constant of moderate size yields a computed matrix that is close to the original weight matrix $A$.

We will address only the problem of computing $A_{\delta}$, but the proposed approach can be extended easily to compute $A_{\delta}^\mathcal{S}$ by introducing the orthogonal projection onto $\mathcal{S}$.

\subsection{A two-level method for the weight matrix stabilisation}
Our solution approach to compute $A_{\delta}$ follows the methods presented in \citep{guglielmi2017matrix,guglielmi2025matrix,guglielmi2023rank} and their extensions \citep{guglielmi2024stabilization,guglielmi2024contractivity}, and it can be sketched as follows. 

Given a target value $ \delta \in\R$, we aim to find a perturbation $\Delta\in \R^{n\times n}$ of minimal Frobenius norm such that 
\[
    \max_{D\in\Omega_m} \mu_2(D(A+\Delta)) =  \delta .
\]
If $\max_{D\in\Omega_m} \mu_2(DA) \le  \delta $, then we simply set $\Delta = 0$, which implies $A_{\delta}=A$. Otherwise we proceed with the computation of $\Delta$, which we rewrite as $\Delta=\eps E$, with
\[
    E \in \mathbb{S}_1 := \{ E\in\R^{n\times n} : \|E\|_F = 1 \}
\]
and where $\eps>0$ denotes the perturbation size. We consider a two-level approach consisting of an {\it inner} and an {\it outer iteration}.
\begin{itemize}
    \item {\bf Inner iteration.} For a fixed $\eps>0$, we solve the optimisation problem
    \begin{equation}\label{eq:inn_it}
        \min_{E\in\mathbb{S}_1} \nu_\eps(E), 
    \end{equation}
    where
    \begin{equation*} \label{eq:mueps}
    \nu_\eps(E) := \max_{D\in\Omega_m} \mu_2 ( D ( A + \eps E ) ).
    \end{equation*}
    We denote the minimiser of \eqref{eq:inn_it} as $E_\star(\eps)$ and we define $f(\eps) =  \nu_\eps \left( E_\star(\eps) \right)$.
    \item {\bf Outer iteration.} We solve the scalar equation with respect to $\eps$
    \begin{equation}\label{eq:feps}
        f(\eps) =  \delta ,
    \end{equation}
    and we tune the parameter $\eps$ by means of a Newton-like technique to solve the problem
    \begin{equation}\label{eq:out_it}
        \min \{\eps>0\ :\ f(\eps)= \delta \},
    \end{equation}
    whose solution is denoted as $\eps_\star$.
\end{itemize}
This yields a matrix $A+\Delta_\star$, with
$\Delta_\star = \eps_\star E_\star(\eps_\star)$, such that 
\[
    f\left( \eps_\star \right) =
    \max_{D\in\Omega_m} \mu_2(D(A+\eps_\star E_\star(\eps_\star))) =  \delta .
\]
This means that for the slightly modified matrix $A + \Delta_\star$, the logarithmic norm of $D (A + \Delta_\star)$ is uniformly bounded from above by $ \delta $, independently of $D \in \Omega_m$.

Numerical experiments on MNIST \citep{deng2012mnist}, FashionMNIST \citep{xiao2017fashion}, CIFAR10 \citep{krizhevsky2009learning} and SVHN \citep{netzer2011reading} datasets show that neural networks with a one layer weight-tied neural ODE, trained using the proposed stabilisation, are more robust against adversarial attacks than those trained using a classical approach.

\section{Inner iteration: minimising the objective function} \label{sec:inner}
In this section, for a fixed $\eps>0$, we focus on the solution of the inner iteration optimisation problem \eqref{eq:inn_it}. In order to minimise $\nu_\eps(E)$ over all matrices $E\in\mathbb{S}_1$, by definition of logarithmic 2-norm \eqref{eq:l2n}, we aim to minimise the functional
\begin{equation}\label{eq:func}
    F_\eps(E) = \frac{1}{2} \sum_{i=1}^{n} (\lambda_i(\Sym(D_\star(A+\eps E)))- \delta )_+^2,
\end{equation}
over all matrices $E\in\mathbb{S}_1$, with
\begin{equation}\label{eq:compDm}
    D_\star = D_\star(\eps,E) := \argmax_{D\in\Omega_m} \mu_2(D(A+\eps E)).
\end{equation}
Therefore the inner iteration consists in solving the following minimisation problem
\begin{equation}\label{eq:min_prob}
    \min_{E\in\mathbb{S}_1} \frac{1}{2} \sum_{i=1}^{n} (\lambda_i(\Sym(D_\star(A+\eps E)))- \delta )_+^2,
\end{equation}
where the computation of $D_\star$ has already been discussed in \citep{guglielmi2024contractivity}.

\begin{remark}
    The matrix $D_\star$ is piecewise constant in $\eps$ and $E$, that is $D_\star$ does not change for small variations of $\eps$ or $E$. Thus, from now on, we consider $D_\star$ as a locally constant matrix with vanishing derivatives.
\end{remark}

We will derive a matrix differential equation whose stationary points are the local minima of the minimisation problem \eqref{eq:min_prob}. We then illustrate how to integrate the differential equation and we discuss the properties of its stationary points.

We will often use the following standard perturbation result for eigenvalues. 
\begin{lemma} {\rm \citep[Section II.1.1]{Kato2013}} \label{lem:eigderiv}
Consider a continuously differentiable symmetric matrix valued function $C : \R \to \R^{n\times n}$. Let $\lambda(t)$ be a simple eigenvalue of $C(t)$ for all $t\in \R$ and let $x(t)$ with $\|x(t) \| =1$ be the associated (right and left) eigenvector. Then $\lambda (t)$ is differentiable with
\begin{equation}\nonumber
\dot\lambda(t) = x(t)^\top \dot{C}(t) x(t) = \langle x(t) x(t)^\top, \dot {C}(t) \rangle_F.
\end{equation}
\end{lemma}

We suppose that $E$ is a smooth matrix valued function $E(t)$ of the independent real variable $t$ such that the eigenvalues $\lambda_i(t)$ of $\Sym(D_\star(A+\eps E(t))$ are simple, for all $t$, with corresponding unit eigenvectors $x_i(t)$ and $D_\star$ is defined in \eqref{eq:compDm}. Then, applying Lemma \ref{lem:eigderiv}, yields
\[
    \dot{\lambda}_i(t) = \eps x_i(t)^\top\left( \frac{D_\star\dot{E}(t)+\dot{E}(t)^\top D_\star}{2} \right)x_i(t) = \eps \frac{x_i(t)^\top D_\star\dot{E}(t)x_i(t) + x_i(t)^\top \dot{E}(t)^\top D_\star x_i(t)}{2}.
\]
Consequently, by omitting the dependence on $t$ for conciseness, we have that
\begin{equation*}
    \frac{1}{\eps}\frac{\text{d}}{\text{d}t} F_\eps(E) = \sum_{i=1}^{n} \gamma_i \left( \frac{x_i^\top D_\star\dot{E} x_i + x_i^\top \dot{E}^\top D_\star x_i}{2} \right) = \sum_{i=1}^{n} \gamma_i z_i^\top \dot{E} x_i = \left\langle \sum_{i=1}^{n} \gamma_i z_i x_i^\top , \dot{E} \right\rangle_F,
\end{equation*}
with $z_i(t) = D_\star x_i(t)$ and $\gamma_i(t) = (\lambda_i(\Sym(D_\star(A+\eps E(t))))- \delta )_+$. Dropping the irrelevant factor $\eps$, which amounts to a rescaling of time, yields that the free gradient of $F_\eps(E)$ is
\begin{equation}\label{eq:free_grad}
    G_\eps(E) = \sum_{i=1}^{n} \gamma_i z_i x_i^\top.
\end{equation}
The admissible steepest descent direction $\dot{E}$ of $F_\eps(E)$ must fulfill the constraint $\|E(t)\|_F=1$ and hence, as shown in \citep[Lemma 2.6]{guglielmi2023rank} and \citep[Lemma 3.2]{guglielmi2024low}, it has to be
\begin{equation}\label{eq:derE}
    \dot{E} = -G_\eps(E) + \mu E, \quad \text{with} \quad \mu = \langle G_\eps(E) , E \rangle_F.
\end{equation}
This yields a differential equation for $E$ which, by construction, preserves the unit Frobenius norm of $E(t)$:
\[
    \frac{1}{2}\dv{t}\|E(t)\|_F^2 = \langle E , \dot{E} \rangle_F = \langle E , -G_\eps(E) + \mu E \rangle_F = - \langle E , G_\eps(E) \rangle_F + \mu \langle E , E \rangle_F = 0.
\]
Moreover, equation \eqref{eq:derE} is a gradient system, i.e. along its solutions the functional $F_\eps(E(t))$ decays monotonically, since the constraint $\|E\|_F=1$ and the Cauchy-Schwarz inequality implies
\begin{equation}\label{eq:comp}
    \frac{1}{\eps}\dv{}{t} F_\eps(E(t)) = \langle G_\eps(E) , \dot{E} \rangle_F = \langle G_\eps(E) , -G_\eps(E) + \mu E \rangle_F = - \|G_\eps(E)\|_F^2 + \langle G_\eps(E) , E \rangle_F^2 \leq 0.
\end{equation}

\begin{remark}
    \label{rem:eigsimple}
    Our approach is based on the eigenvalue perturbation formula of Lemma \ref{lem:eigderiv}, which assumes that the eigenvalues are simple. Situations where eigenvalues are multiple are either non-generic or can happen generically only at isolated times $t$, thus they do not affect the computation after discretisation of the differential equation. Our assumption therefore is not restrictive, since the subset of $\R^{n\times n}$ of matrices with multiple eigenvalues has zero measure \citep{avron1978analytic} and, even if a continuous trajectory runs into a point with multiple eigenvalues, it is highly unlikely that the discretisation of the differential equation captures it. Also, from the practical point of view, we experienced that the numerical integration is not affected by the coalescence of two or more eigenvalues. Thus, from here on, we assume that the eigenvalues of $\Sym(D_\star(A+\eps E))$ are simple.
\end{remark}
%
%

The stationary points of \eqref{eq:derE} are characterised as follows.

\begin{theorem} \label{th:exchar}
    Provided that there is at least an eigenvalue of $\Sym(D_\star(A+\eps E))$ larger than $ \delta $, the following statements are equivalent along solutions of \eqref{eq:derE}:
    \begin{itemize}
        \item[1.] $\dv{}{t} F_\eps(E(t)) = 0$;
        \item[2.] $\dot{E}=0$;
        \item[3.] $E$ is a multiple of $G_\eps(E)=\sum_{i=1}^n \gamma_iz_ix_i^\top$.
    \end{itemize}
\end{theorem}

\begin{proof}
Let $G=G_\eps(E)$, and we recall that
\begin{align}
    &\dot{E} = -G + \langle G , E \rangle_F E , \label{eq:1.1} \\
    &\frac{1}{\eps} \dv{}{t} F_\eps(E) = \langle G , \dot{E} \rangle_F \, . \label{eq:1.2}
\end{align}
Equation \eqref{eq:1.1} and $\|E\|_F = 1$ yields that statement 3 implies statement 2, while equation \eqref{eq:1.2} yields that statement 2 implies statement 1. So it remains to show that statement 1 implies statement 3. Since $G$ is non-zero because we have assumed that at least an eigenvalue is larger than $ \delta $, inequality \eqref{eq:comp} is strict unless $G$ is a real multiple of $E$. Hence, statement 3 implies statement 1.
\end{proof}

In Algorithm \ref{alg_intstep} we provide a schematic description of a single step of the numerical integration of equation \eqref{eq:derE}. We make use of the simple Euler method, since we are only interested in the stationary points of the ODE and not in accurately approximating all the trajectory of the solution, but considering other explicit methods could also be possible.

\begin{algorithm}[ht] 
\DontPrintSemicolon
\KwData{$A$, $E_k \approx E(t_k)$, $f_k = F_\eps(E_k)$, $h_{k}$ (proposed step size), $\eps$, $\theta > 1$}
\KwResult{$E_{k+1}\approx E(t_{k+1})$, $h_{k+1}$}
\Begin{
\nl Initialise the step size by the proposed one, $h=h_{k}$\;
\nl Compute $D_\star^{(k)}=\argmax_{D\in\Omega_m}\mu_2(D(A+\eps E_k))$\;
\nl Compute the eigenpairs $(\lambda_i, x_i)$, $\lambda_i> \delta $ and $\| x_i  \|  = 1$, of $\Sym(D_\star^{(k)}( A+\eps E_k))$\;
\nl Compute $z_i = D_\star^{(k)} x_i$ \;
\nl Compute the gradient $G_k \approx G_\eps(E(t_k))$ according to \eqref{eq:free_grad}\;
\nl Compute $\dot{E}_k \approx \dot{E}(t_k)$ according to \eqref{eq:derE} \;
\nl Initialise $f(h) = f_k$ and $\mathrm{reject=0}$\;
\nl Set $E(h) = E_k + h \dot{E}_k$, and $E(h)=\frac{E(h)}{\|E(h)\|_F}$\;
\While{$f(h) \ge f_k$}{
\nl Compute $D_\star(h)=\argmax_{D\in\Omega_m}\mu_2(D(A+\eps E(h)))$\;
\nl Compute the eigenpairs $(\lambda_i, x_i)$, $\lambda_i> \delta $ and $\| x_i  \|  = 1$, of $\Sym(D_\star(h)( A+\eps E(h)))$, and set $f(h)=F_\eps(E(h))$\;
\If{$f(h) \ge f_k$}{Reduce the step size, $h:=h/\theta$\; \nl Set $E(h) = E_k + h \dot{E}_k$, and $E(h)=\frac{E(h)}{\|E(h)\|_F}$\; \nl Set $\mathrm{reject=1}$}
}
\eIf{$\mathrm{reject=0}$}
{Set $h_{\rm next} := \theta h$}
{Set $h_{\rm next} := h$}
\nl Set the starting values for the next step: $E_{k+1}=E(h)$, and $h_{k+1}=h_{\rm next}$\;
\Return
}
\caption{Integration step from $t_k$ to $t_{k+1}$ for the constrained gradient system (with $\eps$ fixed)
\label{alg_intstep}}
\end{algorithm}

\subsection{Simplification of the inner iteration}

We use the notation that for a vector $d \in \R^n$, $D = \diag(d)\in \R^{n\times n}$ indicates the diagonal matrix with $d$ in the diagonal and viceversa for a matrix $D\in \R^{n\times n}$, $\diagvec(D)\in \R^n$ denotes the vector formed by its diagonal. We notice that after the $k$-th step of the Euler method, with $k>0$, for equation \eqref{eq:derE}, namely lines 2 and 9 in Algorithm \ref{alg_intstep}, the matrix
\begin{equation*}\label{eq:maxD}
    D_\star^{(k)} = \argmax_{D\in\Omega_m} \mu_2(D(A+\eps E_k))
\end{equation*}
has to be updated. 

We let $d_\star^{(k)} = \diag ( D_\star^{(k)} )$. Computing $d_\star^{(k+1)}$ (equivalently the matrix $D_\star^{(k+1)}$) requires the integration of a constrained gradient system, see \citep{guglielmi2024contractivity}, and thus, the overall exact integration of equation \eqref{eq:derE} would be expensive. 
In order to avoid this, we try to approximate $d_\star^{(k+1)}$ by exploiting the property, which we expect, that its entries are either $m$ or 1 (see \citep[Theorem 3.1]{guglielmi2024contractivity}).

We now formalise how this approximation is carried out. Given $E_k\approx E(t_k)$ and $D_\star^{(k)}$, we need to approximate
\begin{equation} \label{eq:Dkp1}
D_\star^{(k+1)}=\argmax_{D\in\Omega_m}\mu_2(D(A+\eps E_{k+1})),
\end{equation}
where $E_{k+1}\approx E(t_{k+1})$. Instead of integrating the constrained gradient system in \citep[Section 3]{guglielmi2024contractivity} for $d = \diagvec(D)$, that is
\begin{equation}\label{eq:Ddot}
    \dot{d} = g := \diagvec(\Sym(zx^\top)),
\end{equation}
where $x$ is the eigenvector associated with the largest eigenvalue of $\Sym(D(A+\eps E_{k+1}))$ and $z=(A+\eps E_{k+1})x$,
we use a semi-combinatorial approach.

We aim to compute a solution $D_\star = \diag (d_\star)$ of \eqref{eq:Dkp1} as follows.
We start setting $d^{[0]} = d_\star^{(k)}$ and then we update it according to the following iterative process starting from $\ell=0$:
\begin{equation}\label{eq:imp_dir_1}
    \left(d^{[\ell+1]}\right)_{i} =
    \begin{cases}
        1 & \quad \text{if $\left(g^{[\ell]}\right)_{i}>0$} \\
        m & \quad \text{if $\left(g^{[\ell]}\right)_{i}<0$} 
    \end{cases},
\end{equation}
where $g^{[\ell]}$ is the gradient defined in \eqref{eq:Ddot}
for $D = \diag(d^{[\ell]})$. If, for some $\ell$, it occurs that $d^{[\ell+1]} = d^{[\ell]}$, the iteration stops, $d^{[\ell]}$ is the sought solution and $D_\star^{(k+1)} = \diag( d^{[\ell]})$.

The idea behind \eqref{eq:imp_dir_1} is that, given the constraint that the entries of $d$ are bounded from below by $m$ and from above by $1$, $d^{[\ell]}$ is optimal if in correspondence of
its maximal entries (equal to $1$), the gradient is positive, and in correspondence with its minimal entries (equal to $m$), the gradient is negative. The number of iterations $maxit$ we perform is typically small ($20$ at most in our experiments). 

The iterative method is outlined in Algorithm \ref{alg_update}.

\begin{algorithm}[ht!] 
\DontPrintSemicolon
\KwData{$A$, $d_\star^{(k)}$, $E_{k+1} \approx E(t_{k+1})$, $\eps$, $maxit$ (default $20$)}
\KwResult{$d_\star^{(k+1)}$}
\Begin{
\nl Set $\ell=0$, $d^{[\ell]}=d_\star^{(k)}$ and $Stop=False$\;
\While{$\ell< maxit\ \text{and}\ Stop=False$}{
\nl Define $d^{[\ell+1]}$ as the update of $d^{[\ell]}$ according to \eqref{eq:imp_dir_1}\;
\nl Compute the eigenvector $x^{[\ell+1]}$ associated with the largest eigenvalue of $\Sym(D^{[\ell+1]}(A+\eps E_{k+1}))$\;
\nl Compute $z^{[\ell+1]}=(A+\eps E_{k+1})x^{[\ell+1]}$\;
\nl Compute the approximate gradient $g^{[\ell+1]}=\diagvec(\Sym(z^{[\ell+1]}(x^{[\ell+1]})^\top))$\;
\nl Check the sign of the approximate gradient $g^{[\ell+1]}$ according to \eqref{eq:imp_dir_1} and set $Stop=True$ if $d^{[\ell+1]}=d^{[\ell]}$ \;
\nl Set $\ell=\ell+1$
}
\eIf{$\ell<maxit$}
{\nl Set $d_\star^{(k+1)}=d^{[\ell]}$\;}
{\nl Compute $D_\star^{(k+1)}$ by integrating the constrained gradient system in \eqref{eq:Ddot}\;}
\Return
}
\caption{Iteration to update $D_\star^{(k+1)}=\diag(d^{(k+1)}_\star)$}
\label{alg_update}
\end{algorithm}

\subsubsection{Illustrative example}

We consider the matrix
\[
A = \left( \begin{array}{rrr}
   -0.39 &  -1.16 &   0.74 \\
    1.14 &   0.96 &   0.15 \\
    0.42 &  -0.14 &  -2.32
    \end{array} \right)
\]
and the perturbation $\eps E(t)$ with $\eps = 0.3$ and
\[
M(t) =
    \left( \begin{array}{rrr}
    -\sin(2 t)/2 & \sin(t^2) & t\\ 
    -t/4 &  -t & \sin(t/2) \\ 
    -\sin(t)^2 & \cos(t) & t
    \end{array} \right), \qquad E(t)=\frac{M(t)}{\|M(t)\|_F}.
\]
We fix $m=0.5$ and compute $\mu(t) = \mu_2 \left( D(t) (A + \eps E(t)) \right)$ in the time interval $[0,1]$ where $D(t)$ is the local extremiser computed by our algorithm. 
Figure \ref{fig:ill} illustrates the descendent behaviour of $\mu(t)$. We use a discretisation stepsize $h=0.05$ and solve the optimisation problem by the simplified inner iteration.
\begin{figure}[ht]
    \centering
    \includegraphics[width=0.5\textwidth]{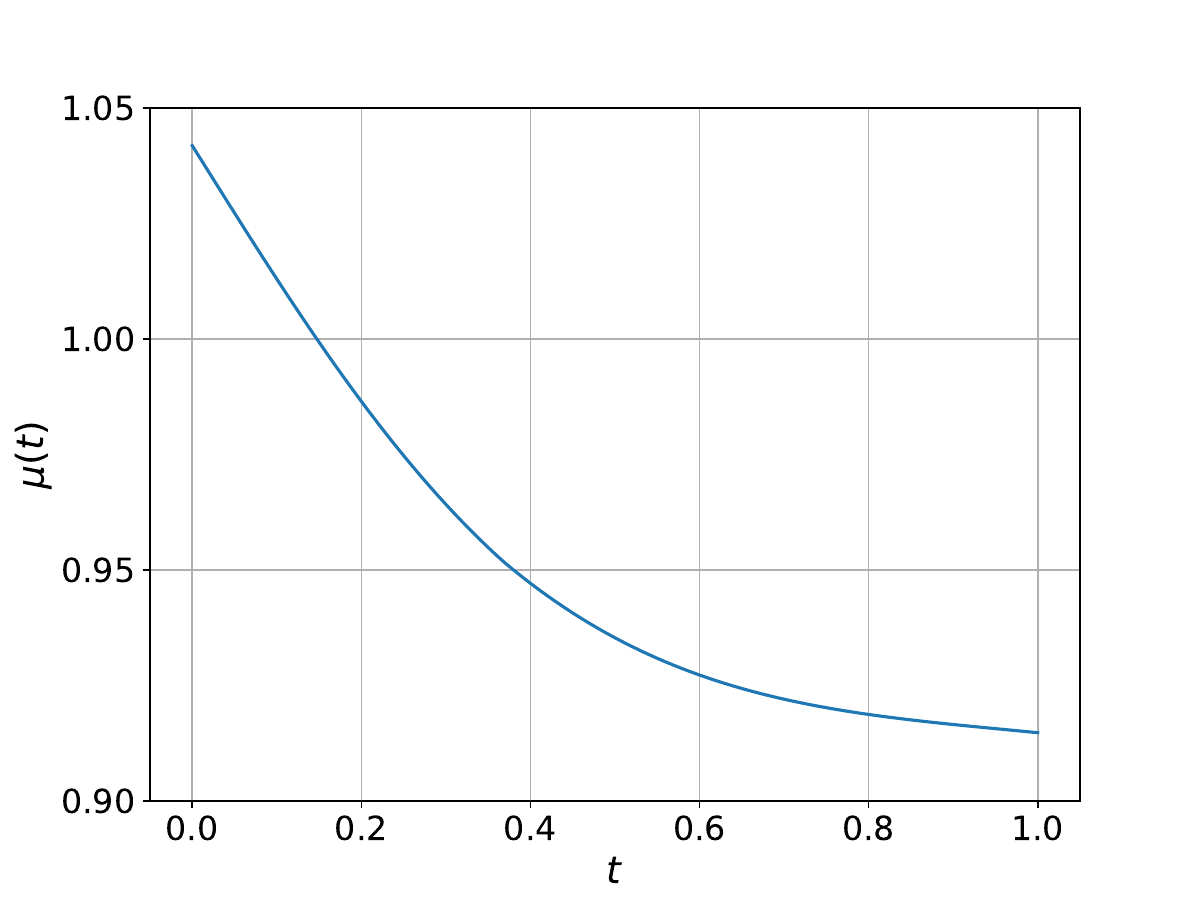}
    \caption{Behaviour of $\mu(t) = \mu(D(t) (A+\eps E(t))$ as $t \in [0,1]$.}
    \label{fig:ill}
\end{figure}
For $t \le 0.3$ the extremiser is $D_1 = \diag\left( m \quad 1 \quad 1 \right) $. At $t=0.45$, the extremiser ceases to exist since the gradient is
\[
    G = \left(-0.2865 \quad  1.0832  \quad  -0.0002 \right) ^\top,
\]
due to the negativity of last entry. At this point, shifting the last entry of $D$ from $m$ to $1$ gives the candidate extremiser $D_2 = \diag\left( m \quad 1 \quad m \right)$. The corresponding gradient is
\[
    G = \left( -0.2804 \quad 1.0830 \quad  -0.0043 \right) ^\top,
\]
which satisfies the necessary condition for extremality so that $D_2$ is indeed an extremiser. From now on, that is for all $t \ge 0.45$, $D_2$ remains an extremiser. However, we note that for $t \in [0.35,0.45]$ both $D_1$ and $D_2$ are extremisers although our algorithm detects only the transition at $t=0.45$ when $D_2$ becomes the unique extremiser. 

\subsection{Two layer vector field}
We now consider a vector field that is a two-layer neural network
\[
    f(x) = \sigma (A_2 \sigma (A_1 x + b_1) + b_2), \quad x\in\R^n.
\]
By using the Lagrange theorem, for any two vectors $x$ and $y$, we get
\begin{equation*}
    f(x) - f(y) = \diag(\sigma'(u_2))A_2 \diag(\sigma'(u_1))A_1 (x-y) = D_2 A_2 D_1 A_1 (x-y),
\end{equation*}
with $D_i = \diag(\sigma'(u_i))$, for suitable vectors $u_1,u_2$. Therefore, with computations analogous to the one layer case, we obtain that the neural ODE \eqref{eq:neuralODE} satisfies the one-sided Lipschitz condition
\[
    \langle f(x) - f(y) , x - y \rangle \le \delta_\star  \|x - y \| ^2, \quad x,y\in\R^n,
\]
with
\[
     \delta_\star := \max_{D_1,D_2\in\Omega_m} \mu_2(D_2 A_2 D_1 A_1).
\]
Given $ \delta \in\R$, $ \delta <\delta_\star $, the goal is now the computation of the smallest matrices $\Delta_1,\Delta_2$ in Frobenius norm such that
\[
    \max_{D_1,D_2\in\Omega_m} \mu_2( D_2 (A_2+\Delta_2) D_1 (A_1+\Delta_1)) =  \delta .
\]
In the following, we will write $\Delta_i = \eps_i E_i$, for $i=1,2$, with $E_i$ a matrix of unit Frobenius norm and $\eps_i$ its amplitude. For simplicity of presentation, we are assuming $\eps_1=\eps_2=\eps$.

Fixed $\eps>0$, the minimisation problem \eqref{eq:min_prob} becomes
\[
    \min_{E_1,E_2\in\mathbb{S}_1} F_\eps(E_1,E_2), \quad \text{with } F_\eps(E_1,E_2) = \frac{1}{2} \sum_{i=1}^{n} (\lambda_i(\Sym(D_\star^{1}(A_1+\eps E_1)D_\star^{2}(A_2+\eps E_2))- \delta )_+^2,
\]
where
\[
    (D_\star^1,D_\star^2) = \argmax_{D_1,D_2\in\Omega_m} \mu_2 (D_2(A_2+\eps E_2)D_1(A_1+\eps E_1)).
\]
Applying Lemma \ref{lem:eigderiv} to $F_\eps(E_1,E_2)$, with $E_i=E_i(t)$ for $i=1,2$, yields
\begin{equation*}
    \frac{1}{\eps}\dv{}{t}F_\eps(E_1(t),E_2(t)) = \left\langle \dot{E}_1 , \sum_{i=1}^n \gamma_i z_i^1 (w_i^1)^\top \right\rangle_F + \left\langle \dot{E}_2 , \sum_{i=1}^n \gamma_i z_i^2 (w_i^2)^\top \right\rangle_F,
\end{equation*}
where $\gamma_i=(\lambda_i(\Sym(D_\star^{1}(A_1+\eps E_1)D_\star^{2}(A_2+\eps E_2)))- \delta )_+$, $x_i$ is the unit eigenvector corresponding to the eigenvalue $\lambda_i(\Sym(D_\star^{1}(A_1+\eps E_1)D_\star^{2}(A_2+\eps E_2)))$, and
\begin{align*}
    z_i^1 &= D_\star^{1} (A_2 + \eps E_2(t))^\top D_\star^2 x_i, \quad w_i^1 = x_i, \\
    z_i^2 &= D_\star^2 x_i, \quad w_i^2 = D_\star^1(A_1+\eps E_1(t))x_i.
\end{align*}
Therefore, for $k=1,2$, the free gradients of $F_\eps$ with respect to $E_k$ are
\[
    G_k = \sum_{i=1}^n \gamma_i z_i^k (w_i^k)^\top,
\]
and the associated constrained gradient systems are
\begin{equation}\label{eq:gs}
    \dot{E}_k = -G_k + \langle G_k , E_k \rangle_F E_k.
\end{equation}

\section{Outer iteration: tuning the size of the perturbation} \label{sec:outer}
In this section, we focus on the solution of the outer iteration optimisation problem \eqref{eq:out_it}. We let $E_\star(\eps)$ of unit Frobenius norm be a local minimiser of the inner iteration \eqref{eq:inn_it}, and we denote by $\lambda_i(\eps)$ and $x_i(\eps)$, for $i=1,\dots,n$, the eigenvalues and unit eigenvectors of $\Sym(D_\star(A+\eps E_\star(\eps)))$. We denote by $\eps_\star$ the smallest value of $\eps$ such that $F_{\eps_\star}(E_\star(\eps_\star))=0$. From the Schur normal form of $\Sym(DA)$, it is trivial to construct a diagonal perturbation of the diagonal Schur factor that shifts the eigenvalues larger than $ \delta $ to the value of $ \delta $. Hence, it follows that $\eps_\star$ exists, with
\[
    \eps_\star \le \left(\sum_{i=1}^n (\lambda_i(\Sym(D_\star A)) -  \delta ))_+^2\right)^\frac{1}{2}.
\]    
To determine $\eps_\star$, we are thus left with a one-dimensional root-finding problem, for which a variety of standard methods, such as bisection and the like, could be employed. In the following we derive a Newton-like algorithm, for which we need to impose an extra assumption.

\begin{assumption}\label{ass:1}
    For $\eps<\eps_\star$, we assume that the eigenvalues $\lambda_i(\eps)$ of $\Sym(D_\star(A+\eps E_\star(\eps)))$ are simple. Moreover, $E_\star(\eps)$, $\lambda_i(\eps)$ and $x_i(\eps)$ are assumed to be smooth functions of $\eps$.
\end{assumption}

\begin{remark}
     Assumption \ref{ass:1} is important because, if this is not satisfied, the standard derivative formula for simple eigenvalues in Lemma \ref{lem:eigderiv}, which we use, would not hold. In principle, this assumption limits the set of matrices that the algorithm is able to stabilise, since there will be cases where the closest stabilised matrix has multiple eigenvalues. But, such matrices form a subset of $\R^{n\times n}$ with zero measure \citep{avron1978analytic}, so this is a nongeneric situation.
\end{remark}

%
%

For $\eps<\eps_\star$, the following result provides a cheap formula for the computation of the derivative of $f(\eps) = F_\eps(E_\star(\eps))$, which is the basic tool in the construction of the outer iteration of the method. This is expressed in terms of the free gradient of $F_\eps$ at $E_\star(\eps)$:
\[
    G(\eps) = \sum_{i=1}^n \gamma_i(\eps) z_i(\eps) x_i(\eps)^\top,
\]
with $\gamma_i(\eps) = (\lambda_i(\Sym(D_\star(A+\eps E_\star(\eps))))- \delta )_+$ and $z_i(\eps) = D_\star x_i(\eps)$.

\begin{lemma} \label{lem:fepsder}
For $\eps<\eps_\star$, under Assumption \ref{ass:1}, the function $f(\eps)=F_\eps(E_\star(\eps))$ is differentiable, and its derivative equals
\begin{equation}\label{eq:dfeps}
    f'(\eps) = - \|G(\eps)\|_F.
\end{equation}
\end{lemma}

\begin{proof}
    See \citep[Lemma 3.5]{guglielmi2017matrix}.
\end{proof}


\subsection{An algorithm to approximate the size of the perturbation}
Since the eigenvalues are assumed to be simple, the function $f(\eps)$ has a double zero at $\eps_\star$ because it is a sum of squares, and hence it is convex for $\eps\le\eps_\star$. This means that we may approach $\eps_\star$ from the left by means of the classical Newton iteration
\[
    \eps_{k+1} = \eps_{k} - \frac{f(\eps_k)}{f'(\eps_k)}, \quad k = 0,1,\dots
\]
which is such that $|\eps_{k+1}-\eps_\star| \approx \frac{1}{2} |\eps_k-\eps_\star|$ and $\eps_{k+1}<\eps_\star$ if $\eps_k < \eps_\star$.

The estimate $|\eps_{k+1}-\eps_\star| \approx \frac{1}{2} |\eps_k-\eps_\star|$ is due to the fact that the Newton method is applied to a function with a double zero. The convexity of the function to the left of $\eps_\star$ guarantees the monotonicity of the sequence and its boundedness.

In Algorithm \ref{alg_approxepsstar} we illustrate how the inner level and the outer level of our optimisation approach have to be combined in order to get the desired approximation of $\eps_\star$. It is important in lines 1 and 4 to choose properly the initial datum of the constrained gradient system \eqref{eq:derE} integration in order to get the best performance of the algorithm.

\begin{algorithm}[ht] 
\DontPrintSemicolon
\KwData{A tolerance $\mbox{tol}>0$ and an initial $\eps_0<\eps_\star$}
\KwResult{$\eps_\star$ and $E(\eps_\star)$}
\Begin{
\nl Integrate the constrained gradient system \eqref{eq:derE} with fixed perturbation size $\eps_0$\;
\nl Compute $f_0=f(\eps_0)$ and $f'_0=f'(\eps_0)$ as in \eqref{eq:feps} and \eqref{eq:dfeps}, and set $k=0$\;
\While{$f_k \ge \mbox{tol}$}{
\nl Set $\eps_{k+1} = \eps_k - \frac{f_k}{f'_k}$\;
\nl Integrate the constrained gradient system \eqref{eq:derE} with fixed perturbation size $\eps_{k+1}$\;
\nl Compute $f_{k+1}=f(\eps_{k+1})$ and $f'_{k+1}=f'(\eps_{k+1})$ as in \eqref{eq:feps} and \eqref{eq:dfeps}, and set $k=k+1$\;
}
\nl Set $\eps_\star=\eps_k$ and $E(\eps_\star) = E(\eps_k)$\;
\Return
}
\caption{Basic algorithm for approximating $\eps_\star$
\label{alg_approxepsstar}}
\end{algorithm}

\begin{itemize}
    \item At perturbation size $\eps_0$, the initial datum is simply chosen to be the opposite of the normalised gradient $-G_{\eps_0}(E)/\|G_{\eps_0}(E)\|_F$ of the functional \eqref{eq:func} evaluated at $E=0$, the zero matrix.
    \item At perturbation size $\eps_{k+1}$, for $k\ge0$, the initial datum is computed in the following way. Given a minimum $E(\eps_k)$ at perturbation size $\eps_k$, the free gradient system
    \begin{equation}\label{eq:Edot_free}
        \dot{E}(t) = -G(E(t)), \quad t\ge0,
    \end{equation}
    with initial datum $E(0)=\eps_k E(\eps_k)$, is solved numerically until $\|E(t)\|_F = \eps_{k+1}$. The result of the numerical integration is then normalised in Frobenius norm and set as the initial datum of the constrained gradient system \eqref{eq:derE} integration at perturbation size $\eps_{k+1}$. This choice guarantees the overall continuity of the functional \eqref{eq:func} throughout its minimisation at different perturbation sizes $\eps$, which is essential to get the best performance of the overall procedure.
\end{itemize}
    
The gradient system \eqref{eq:Edot_free} is integrated with the Euler method -- other explicit methods could also be used -- with constant step size $h$:
\[
    E_{\ell+1} = E_\ell - hG(E_\ell), \quad \ell=0,1,\dots
\]
At each integration step, the norm of the matrix $E(t)$ increases in order to move the eigenvalues below $ \delta $, that is $\|E_\ell\|_F<\|E_{\ell+1}\|_F$, for $\ell\ge0$. Therefore the numerical integration is stopped when $\|E_{\ell^\star+1}\|_F \ge \eps_{k+1}$ for some $\ell^\star$, and the last integration step size is tuned using the Newton method to get the equality. This is achieved as follows.
    
We define the function
\[
    g(h) = \|E_{\ell^\star+1}\|_F^2 - \eps_{k+1}^2 = \|E_{\ell^\star} - hG(E_{\ell^\star})\|_F^2 - \eps_{k+1}^2,
\]
and we aim to find its root $h^\star$ such that $g(h^\star)=0$. Its derivative is simply
\[
    g'(h) = - 2 \sum_{i,j=1}^n (E_{\ell^\star} - h G(E_{\ell^\star}))_{ij} G(E_{\ell^\star})_{ij}.
\]
Then, starting from $h_0=h/2$, we perform the Newton method
\[
    h_{i+1} = h_i - \frac{g(h_i)}{g'(h_i)}, \quad i = 0,1,\dots
\]
until we get the sought $h^\star$ and, consequently, $E_{\ell^\star+1} = E_{\ell^\star} - h^\star G(E_{\ell^\star})$. Eventually, $E_{\ell^\star+1}$ is normalised in Frobenius norm and set as the initial datum of the constrained gradient system \eqref{eq:derE} integration at perturbation size $\eps_{k+1}$.

\subsection{Extension to the two layer vector field}
We denote by $E_\star^1(\eps),E_\star^2(\eps)$ the stationary points of the gradient systems \eqref{eq:gs} and  we look for the smallest zero $\eps_\star$ of
\[
    f(\eps) = F_\eps(E_\star^1(\eps), E_\star^2(\eps)).
\]
As in the one layer case, to determine $\eps_\star$, we are left with a one-dimensional root-finding problem, for which we derive a Newton-like algorithm. We suppose that the following assumption, analogous to Assumption \ref{ass:1}, holds.

\begin{assumption}\label{ass:2}
    For $\eps<\eps_\star$, we assume that the eigenvalues $\lambda_i(\eps)$ of $\Sym(D_\star^{2}(A_2+\eps E_\star^2(\eps))D_\star^{1}(A_1+\eps E_\star^1(\eps)))$ are simple. Moreover, $E_\star^k(\eps)$, $\lambda_i(\eps)$ and the associated unit eigenvectors $x_i(\eps)$ are assumed to be smooth functions of $\eps$.
\end{assumption}

If $G_1(\eps),G_2(\eps)$ are the gradients of $F_\eps$ with respect to $E_\star^1(\eps),E_\star^2(\eps)$, then we have a result analogous to Lemma \ref{lem:fepsder}.

\begin{lemma}
For $\eps<\eps_\star$, under Assumption \ref{ass:2}, the function $f(\eps)=F_\eps(E_\star^1(\eps), E_\star^2(\eps))$ is differentiable, and its derivative equals
\begin{equation*}
    f'(\eps) = - ( \|G_1(\eps)\|_F + \|G_2(\eps)\|_F ).
\end{equation*}
\end{lemma}

\section{Numerical experiments} \label{sec:num_exp}

We test our approach to increase the robustness of an image classifier based on a neural network that includes a neural ODE in its architecture.
We evaluate performance on four standard benchmark datasets, with representative samples shown in Figure \ref{fig:dataset-samples}. MNIST \citep{deng2012mnist} contains handwritten digits and serves as a fundamental benchmark for image processing systems. FashionMNIST \citep{xiao2017fashion} consists of Zalando product images across 10 clothing categories, providing a more challenging alternative to MNIST. CIFAR10 \citep{krizhevsky2009learning} comprises natural images from 10 classes (airplanes, cars, birds, cats, deer, dogs, frogs, horses, ships, and trucks), widely used for computer vision algorithm development. SVHN (Street View House Numbers) \citep{netzer2011reading} presents real-world digit recognition challenges using cropped house numbers from Google Street View images, offering significantly greater complexity than MNIST while maintaining a similar digit-focused task structure.

\begin{figure}[ht]
    \centering
    \begin{tabular}{cc cc}
        \rotatebox{90}{$\hspace{1.5em}$MNIST} & \includegraphics[width=.4\textwidth]{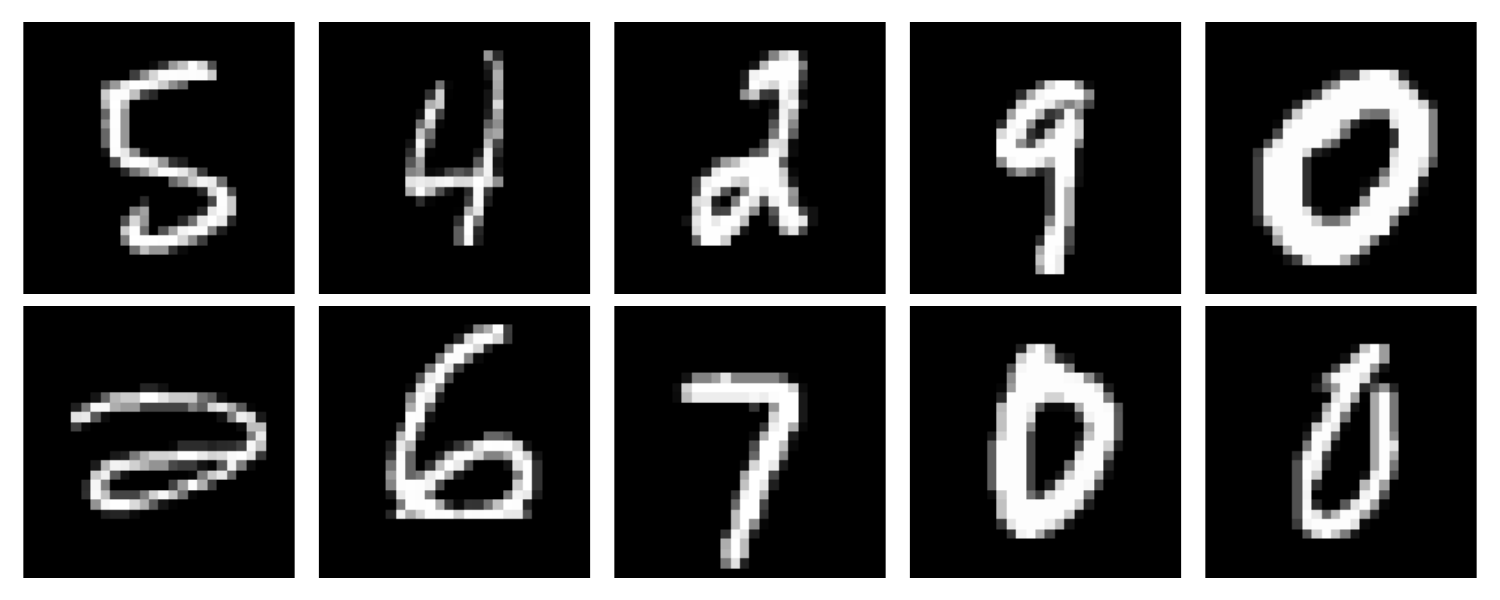} &
        \rotatebox{90}{$\hspace{1em}$FMNIST} & \includegraphics[width=.4\textwidth]{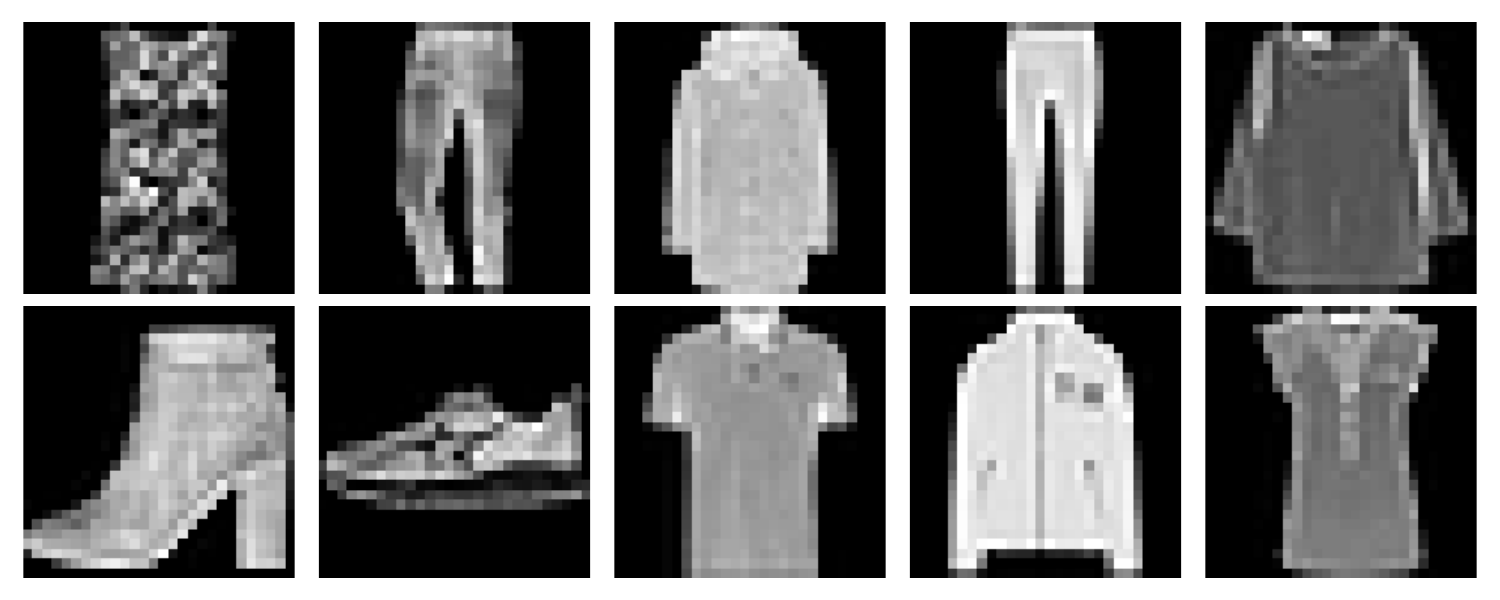} \\
        \rotatebox{90}{$\hspace{1em}$CIFAR10} & \includegraphics[width=.4\textwidth]{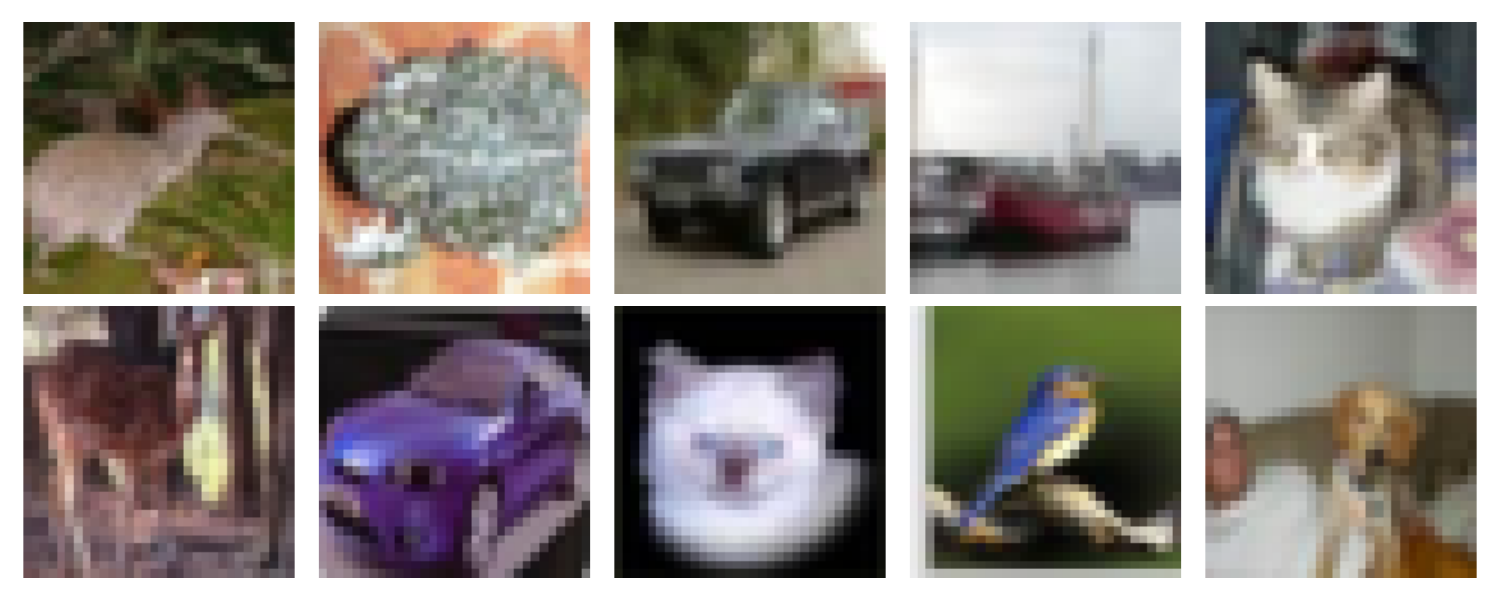} &
        \rotatebox{90}{$\hspace{2em}$SVHN} & \includegraphics[width=.4\textwidth]{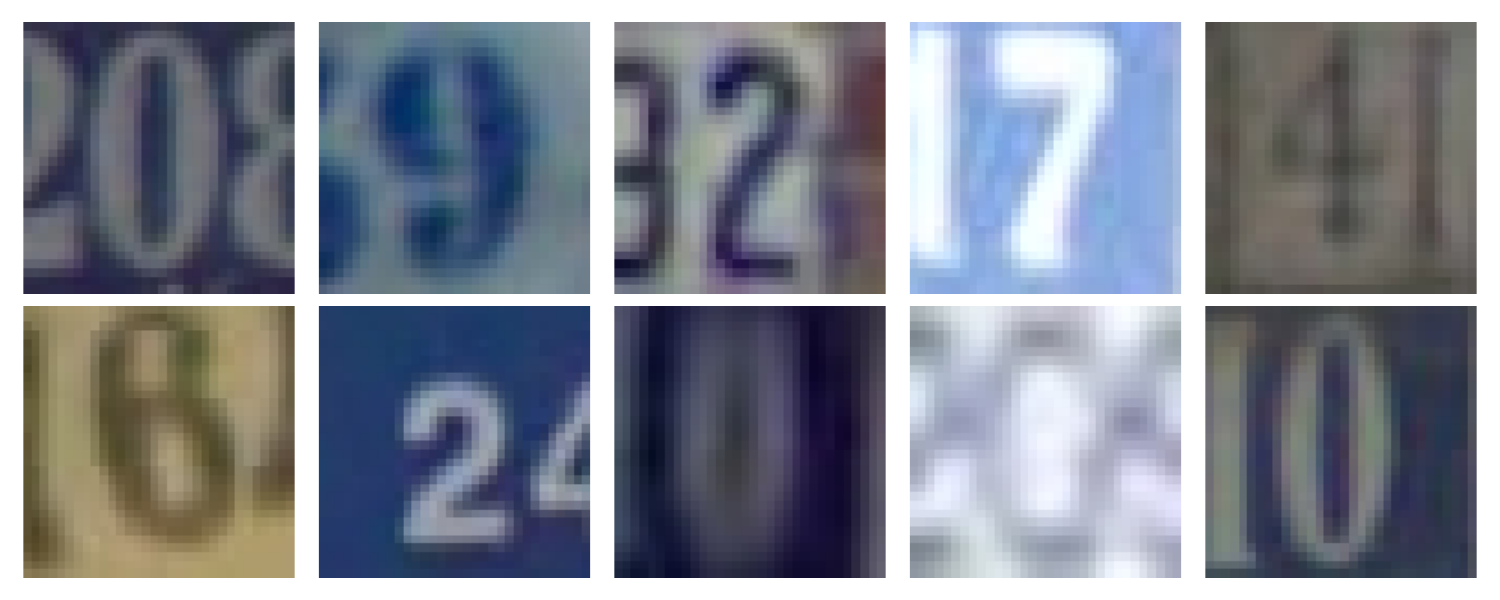} 
    \end{tabular}
    \caption{Sample images from the datasets MNIST, FashionMNIST, CIFAR10, and SVHN.}
    \label{fig:dataset-samples}
\end{figure}


%

\subsection{Architectures and training details}

\subsubsection{MNIST and FashionMNIST datasets}
MNIST and FashionMNIST consist of 70000 $28\times28$ grayscale images (60000 training images and 10000 testing images), that is, vectors of length 784 after vectorisation, grouped in 10 classes. We consider a simple neural network made up of the following blocks:
\begin{itemize}
    \item[(a)] a downsampling affine layer that reduces the dimension of the input from 784 to 64, i.e. a simple transformation of the kind $y = A_1x + b_1$,
    where $x\in\R^{784}$ is the input, $y\in\R^{64}$ is the output, and $A_1\in\R^{64\times 784}$ and $b_1\in\R^{64}$ are the parameters;
    \item[(b)] a neural ODE block that models the feature propagation,
    \[
        \begin{cases}
            \dot{x}(t) = \sigma\left(Ax(t)+b\right), \qquad t\in[0,1],\\
            x(0) = y,
        \end{cases}
    \]
    whose initial value is the output of the previous layer, where $x:[0,1]\to\R^{64}$ is the feature vector evolution function, $A\in\R^{64\times 64}$ and $b\in\R^{64}$ are the parameters, and $\sigma$ is a smoothed LeakyReLU activation function defined as follows:    {
    \[
        \sigma(z) =
        \begin{cases}
            z, \qquad &\mbox{if } z \geq 0,\\
            \tanh{z}, \qquad &\mbox{if }  -\bar{z} \le z < 0,\\
            \alpha z + \beta, \qquad &\mbox{otherwise},
        \end{cases}
    \]
    where $\bar{z}>0$ is such that $\tanh'{(\pm\bar{z})}=\alpha=0.1$ and $\beta\in\R$ such that $\alpha (-\bar{z}) + \beta = \tanh{(-\bar{z})}$ (see Figure \ref{fig:custom_act_fun});
    }
    \begin{figure}[ht]
        \centering
        \includegraphics[width=0.5\textwidth]{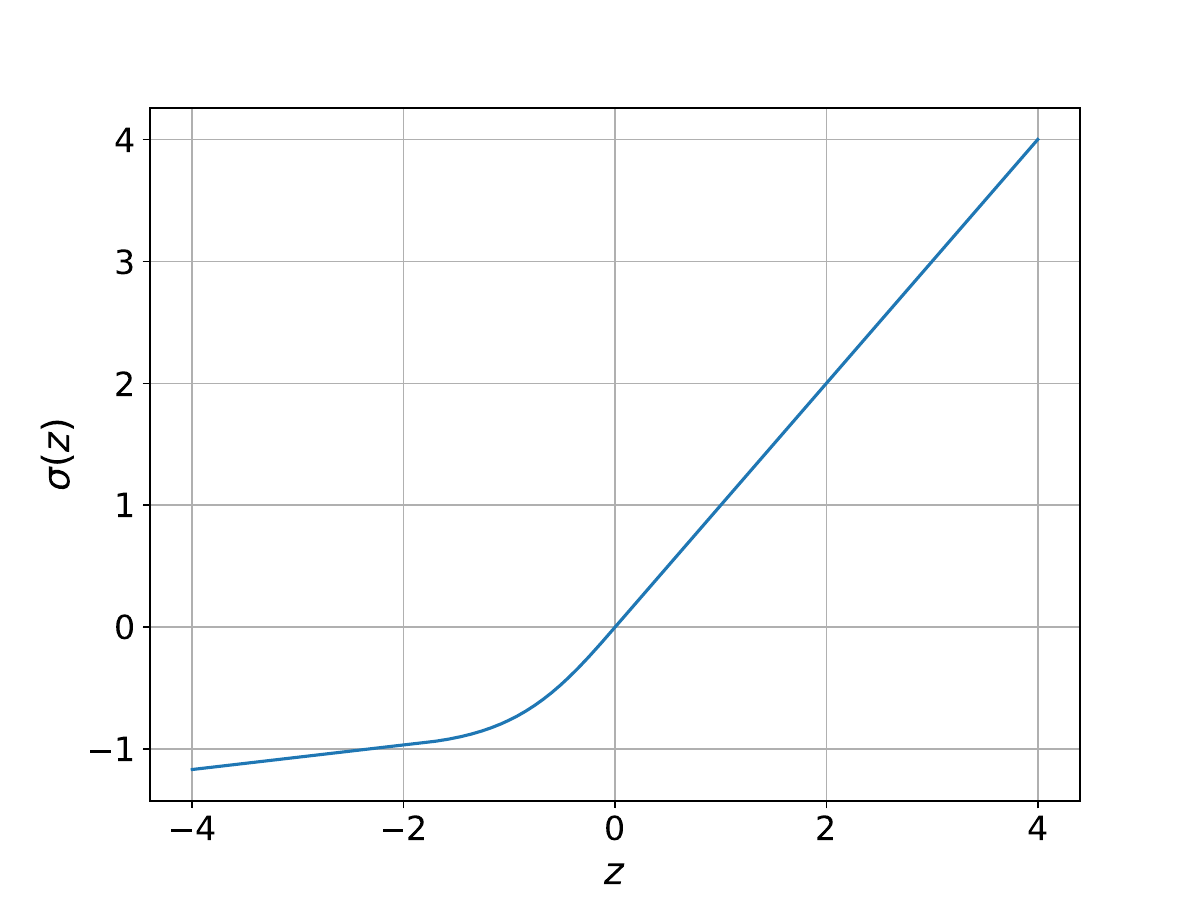}
        \caption{Custom activation function: a smoothed LeakyReLU with minimal slope $\alpha=0.1$.}
        \label{fig:custom_act_fun}
    \end{figure}
    \item[(c)] a final classification layer that reduces the dimension of the input from 64 to 10, followed by the softmax activation function
    \[
        x_{\text{out}} = \mbox{softmax}\left(A_2x(1)+b_2\right),
    \]
    where $x(1)\in\R^{64}$ is the output of the neural ODE block, $A_2\in\R^{10\times 64}$ and $b_2\in\R^{10}$ are the parameters, and $x_{\text{out}}$ is the output vector whose component $i$ is the probability that the input $x$ belongs to the class $i$. We recall that $\mbox{softmax}$ is a vector-valued function that maps the vector $x$ into the vector $\mbox{softmax}(x)=e^x\|e^x\|_1^{-1}$, where exponentiation is done entrywise.
\end{itemize}
If we call $\ell_1$ the transformation at point (a), $\phi_1$ the flow of the differential equation at point (b) at time 1, and $\ell_2$ the transformation at point (c), then the neural network model $\phi$ can be written as
\[
    \phi = \ell_2 \circ \phi_1 \circ \ell_1 = \mbox{softmax}\left(A_2 \phi_1 \left(A_1 \bullet + b_1\right) + b_2\right),
\]
and it holds that, for any vectors $x,y$,
\[
    \|\phi(x)-\phi(y) \|  \le L_{\text{softmax}} \|A_2 \|  L_{\phi_1} \|A_1 \|  \|x-y \| ,
\]
where $L_{\text{softmax}}$ and $L_{\phi_1}$ are the Lipschitz constants of $\mbox{softmax}$ and $\phi_1$ respectively. We note that, it is easy to show that $L_{\text{softmax}} = 1$, and that
\[
    L_{\phi_1} \le \exp ( \max_{D\in\Omega_m} \mu_2(DA) ).
\]
If we set $\|A_2 \| =1$, then
\[
    \|\phi(x)-\phi(y) \|  \le L_{\phi_1} \|A_1 \|  \|x-y \|  \le \exp \left( \max_{D\in\Omega_m} \mu_2(DA) \right) \|A_1 \|  \|x-y \| ,
\]
i.e. the Lipschitz constant of $\phi$ is bounded by $\exp ( \max_{D\in\Omega_m} \mu_2(DA) ) \|A_1 \| $, which we are able to tune as we like thanks to the proposed two-level method. In particular, setting it equal to a constant of moderate size, we make sure that small perturbations in the input yield only small changes in the output. This increases the robustness of the model $\phi$.

We now build two alternative versions of the above-mentioned model on MNIST and FashionMNIST datasets.
\begin{itemize}
    \item We train the first one by using the classical stochastic gradient descent for 70 epochs leading to the classical model, and we call it \texttt{ODEnet} from here on. There is no constraint on $\|A_1 \| $ so far.
    \item Then, we build the second version of the model as follows.
    \begin{enumerate}
        \item We retain the parameters of the classical one and we fix $\|A_1 \| $ to the value obtained after training in the previous step.
        \item We apply the two-level method to the weight matrix $A$ in (b) to impose that $\exp ( \max_{D\in\Omega_m} \mu_2(DA) ) \|A_1 \| $ has moderate size. Specifically, given a target value $ \delta \in\R$, we compute $\eps_\star>0$ and $E_\star(\eps_\star)\in\mathbb{S}_1$ such that
        \begin{equation}\label{eq:pert}
            \hat{A} = A + \eps_\star  E_\star(\eps_\star)
        \end{equation}
        realizes the condition $\max_{D\in\Omega_m}\mu_2(D \hat{A}) =  \delta $.
        \item Once fixed the parameter $\hat{A}$ in place of $A$ in (b), we train the second version of the model by computing all other parameters by means of the classical stochastic gradient descent for 70 epochs, with $\|A_1 \| >1$ fixed as in point 1, and $\|A_2 \| =1$.
    \end{enumerate}
    This is the proposed stabilised model, and we call it \texttt{stabODEnet} from here on. 
\end{itemize}


\subsubsection{CIFAR10 and SVHN datasets}
CIFAR10 contains 60,000 RGB images of size $3\times32\times32$ (50,000 for training and 10,000 for testing) distributed across 10 classes, while SVHN comprises 99,289 RGB images of the same dimensions (73,257 for training and 26,032 for testing) also spanning 10 classes.

To accommodate the increased complexity of CIFAR10 and SVHN compared to the simpler grayscale datasets, we modify our model architecture by incorporating convolutional operators in the initial downsampling component. Specifically, the adapted architecture is defined as follows. The first downsampling block transforms the input $x\in\R^{3\times32\times32}$ into a feature vector $y\in\R^{64}$ through three convolutional layers, specifically,
\[
    y = A_3\sigma(A_2\sigma(A_1x+b_1)+b_2)+b_3,
\]
where $A_1,A_2,A_3$ are convolution operators with kernel size 5 and stride 2, $b_1,b_2,b_3$ are the corresponding biases and $\sigma(z) = \text{ReLU}(z):=\max\{0,z\}$. This is followed by a neural ODE layer and a classification layer, with the chosen smoothed LeakyReLU activation function in Figure \ref{fig:custom_act_fun}, as before.

We note that analogous computations as the one above yield 
\[
    \|\phi(x)-\phi(y) \|  \le L_{\text{softmax}} \|A_4 \|  L_{\phi_1} \|A_3 \|  L_{\sigma} \|A_2 \|  L_{\sigma} \|A_1 \|  \|x-y \| ,
\]
for any vectors $x,y$, where $\phi$ denotes the neural network model. Thus, if we set $\|A_4 \| =1$, we obtain the bound 
\begin{align*}
    \|\phi(x)-\phi(y) \|  &\le \|A_1 \|  \|A_2 \|  \|A_3 \|  L_{\phi_1} \|x-y \|  \\ &\le \exp ( \max_{D\in\Omega_m} \mu_2(DA) ) \|A_1 \|  \|A_2 \|  \|A_3 \|  \|x-y \| .
\end{align*}
In other words, the Lipschitz constant of $\phi$ is bounded in this case by
\[
    \exp ( \max_{D\in\Omega_m} \mu_2(DA) ) \|A_1 \|  \|A_2 \|  \|A_3 \| ,
\]
which we are able to tune as we like thanks to the proposed two-level method. In particular, setting it equal to a constant of moderate size, we make sure that small perturbations in the input yield only small changes in the output. This increases the robustness of the model $\phi$.

In order to verify this in practice, we now proceed as done for the MNIST and FashionMNIST datasets. 
\begin{itemize}
    \item We first train the neural network using the Adam optimiser for 140 epochs for CIFAR10 and for 100 epochs for SVHN, leading to standard models which we call  \texttt{ODEnet}. There are no constraints on $\|A_1 \| $, $\|A_2 \| $ and $\|A_3 \| $ so far.
    \item Then, we build a second, robust version of these models, proceeding as follows.
    \begin{enumerate}
        \item We fix $\|A_1 \| $, $\|A_2 \| $ and $\|A_3 \| $ to the values obtained after training in the previous step.
        \item We apply the two-level method to the weight matrix $A$ defining the neural ODE in (b) to compute $\eps_\star>0$ and $E_\star(\eps_\star)\in\mathbb{S}_1$ such that
        \begin{equation}\label{eq:pert1}
            \hat{A} = A + \eps_\star  E_\star(\eps_\star)
        \end{equation}
        realises the condition $\max_{D\in\Omega_m}\mu_2(D \hat{A}) =  \delta $.
        \item Once fixed the parameter $\hat{A}$ in place of $A$ in (b), we train the model again computing all other parameters by means of the Adam optimiser for 140 epochs for CIFAR10 and for 100 epochs for SVHN, with $\|A_1 \| >1$, $\|A_2 \| >1$ and $\|A_3 \| >1$ fixed as in point 1. and $\|A_4 \| =1$.
    \end{enumerate}
    This leads to the stabilised model \texttt{stabODEnet}.
\end{itemize}

\begin{remark}
    The stabilised model requires approximately twice the training time of the classical model due to its two-stage training procedure. First, we train the classical model, then apply the operation in \eqref{eq:pert} or \eqref{eq:pert1}, and finally train the stabilised model. The primary computational overhead stems from computing $E_\star(\eps_\star)$ in \eqref{eq:pert} or \eqref{eq:pert1}, with cost scaling according to the dimension of the matrix $A$. For our experimental setting with $A\in\mathbb{R}^{64\times 64}$, this computation requires only a few seconds and is negligible compared to the overall training time, confirming that the stabilised model's training time is effectively twice that of the classical model. However, for larger matrices such as $A\in\mathbb{R}^{1000\times 1000}$, computing $E_\star(\eps_\star)$ requires several minutes, representing a more substantial computational overhead.
\end{remark}

In our experiments, we compare the accuracy of the models, i.e. the percentage of correctly classified testing images, as a function of a parameter $\eta>0$. In particular, we consider the Fast Gradient Sign Method (FGSM) and the Fast Gradient Method (FGM) adversarial attacks \citep{biggio2013evasion,goodfellow2014explaining,szegedy2013intriguing} and we denote by $\eta$ either
\begin{itemize}
    \item the size in the $\ell_\infty$-norm of the FGSM attack to each testing image or
    \item the size in the $\ell_2$-norm of the FGM attack to each testing image,
\end{itemize}
i.e. if the vector $x$ is a testing image and $\delta_x$ is the attack computed with the FGSM such that $\|\delta_x\|_\infty=1$ or with the FGM such that $\|\delta_x\|_2=1$, then $x+\eta\delta_x$ is the resulting perturbed testing image.

\subsection{Parameter selection} 
We employ a four-fold cross-validation strategy to select the optimal value of the parameter $ \delta $ for our \texttt{stabODEnet} model. This approach is motivated by our expectation of a trade-off between model accuracy and adversarial robustness when varying $ \delta $. Specifically, we anticipate that higher values of $ \delta $ will yield improved clean accuracy, but potentially compromise the model's stability and robustness against adversarial perturbations, while lower values should enhance robustness at the expense of standard performance. Furthermore, we expect the extent and nature of this trade-off to be dataset-dependent, necessitating careful parameter tuning for each specific problem domain.

Our cross-validation procedure operates as follows. We partition the training dataset into four equal batches. In each fold, we use three batches for model training while reserving the fourth batch for validation across different candidate values of $ \delta $. This process is repeated four times, with each batch serving as the validation set exactly once. We then select the $ \delta $ value that achieves the highest average validation performance across all four folds for each specific attack magnitude $\eta$.

Figure \ref{fig:inf_norm} illustrates the average validation accuracy of \texttt{stabODEnet} over the four cross-validation folds for MNIST, FashionMNIST, CIFAR10 and SVHN datasets, respectively, when $\eta$ is the size of the FGSM attack. Figure \ref{fig:2_norm} illustrates the same when $\eta$ is the size of the FGM attack. These results confirm our expected accuracy-robustness trade-off and demonstrate how the optimal parameter choice varies both with attack strength and across different datasets.

\begin{figure}[ht]
    \centering
    \begin{subfigure}{0.425\textwidth}
        \centering
        \includegraphics[width=\textwidth]{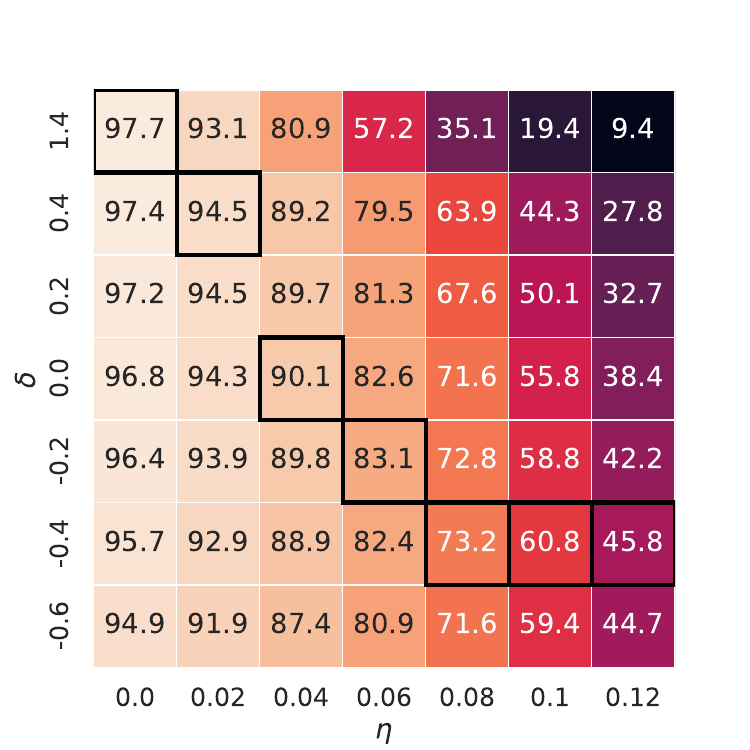}
        \caption{MNIST}
    \end{subfigure}
    \begin{subfigure}{0.425\textwidth}
        \centering
        \includegraphics[width=\textwidth]{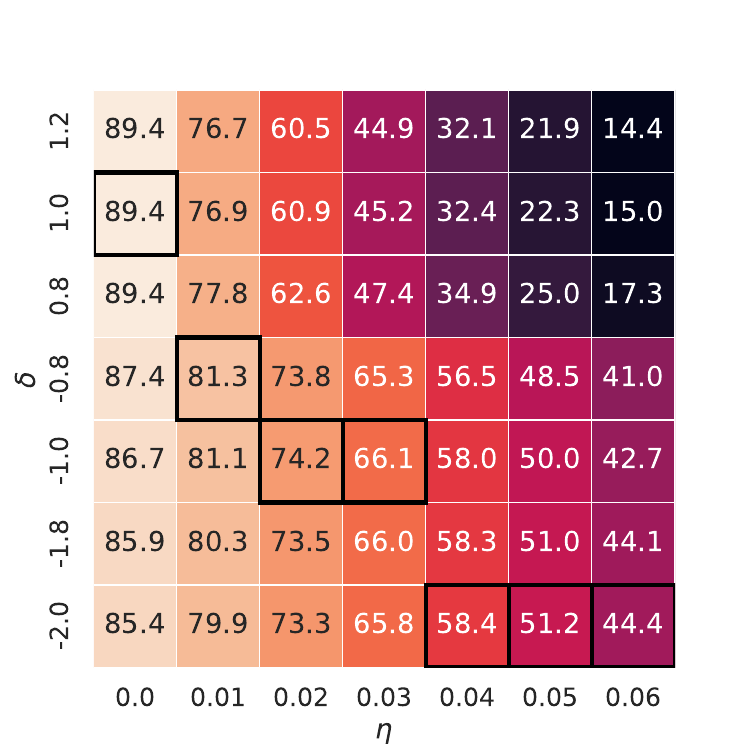}
        \caption{FashionMNIST}
    \end{subfigure}
    \\
    \begin{subfigure}{0.425\textwidth}
        \centering
        \includegraphics[width=\textwidth]{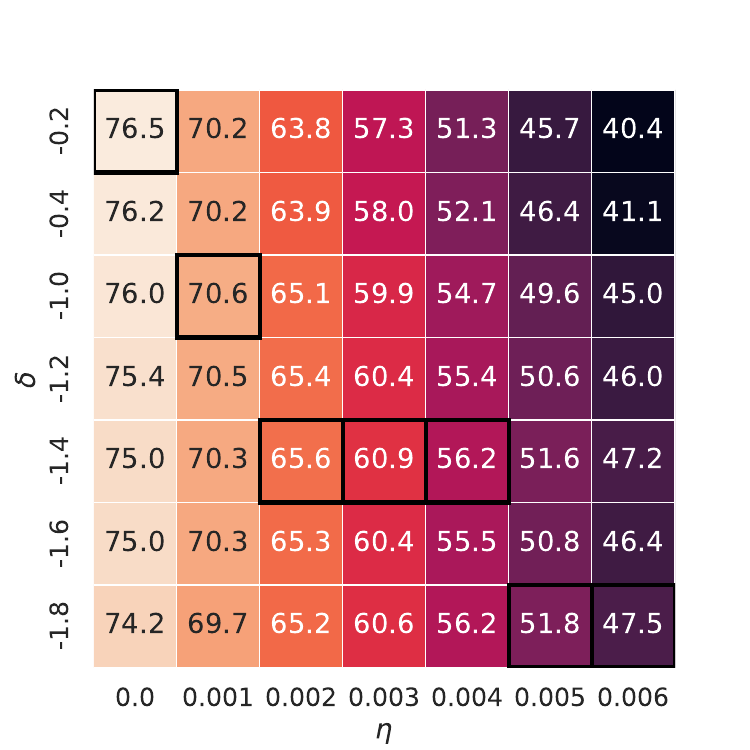}
        \caption{CIFAR10}
    \end{subfigure}
    \begin{subfigure}{0.425\textwidth}
        \centering
        \includegraphics[width=\textwidth]{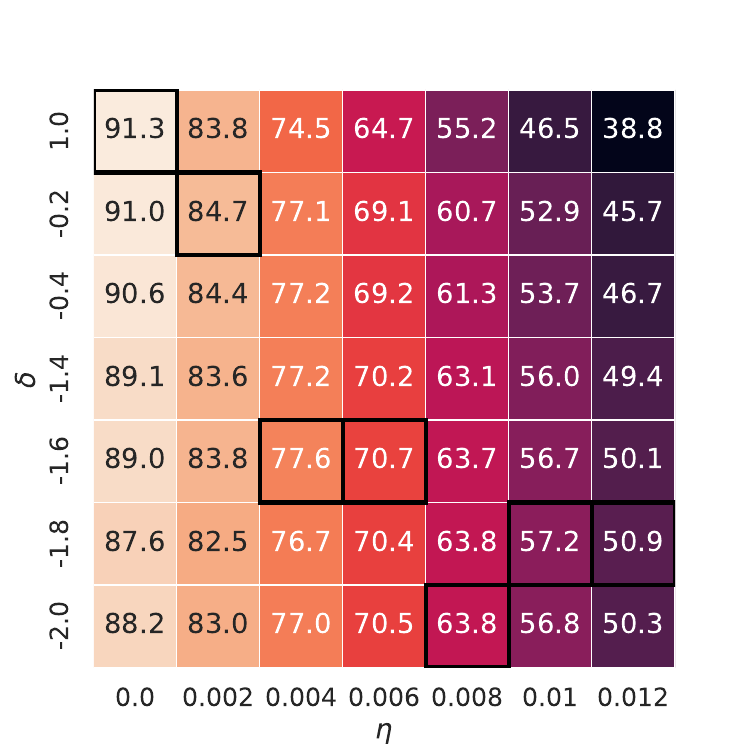}
        \caption{SVHN}
    \end{subfigure}
    \caption{Average validation accuracy of \texttt{stabODEnet} across four cross-validation folds for various $ \delta $ values, plotted as a function of FGSM attack magnitude $\eta$ (measured in $\ell_\infty$ norm). Each row represents a different $ \delta $ setting, and the highest average validation accuracy, rounded to one decimal place, for each attack magnitude is highlighted.}
    \label{fig:inf_norm}
\end{figure}

\begin{figure}[ht]
    \centering
    \begin{subfigure}{0.425\textwidth}
        \centering
        \includegraphics[width=\textwidth]{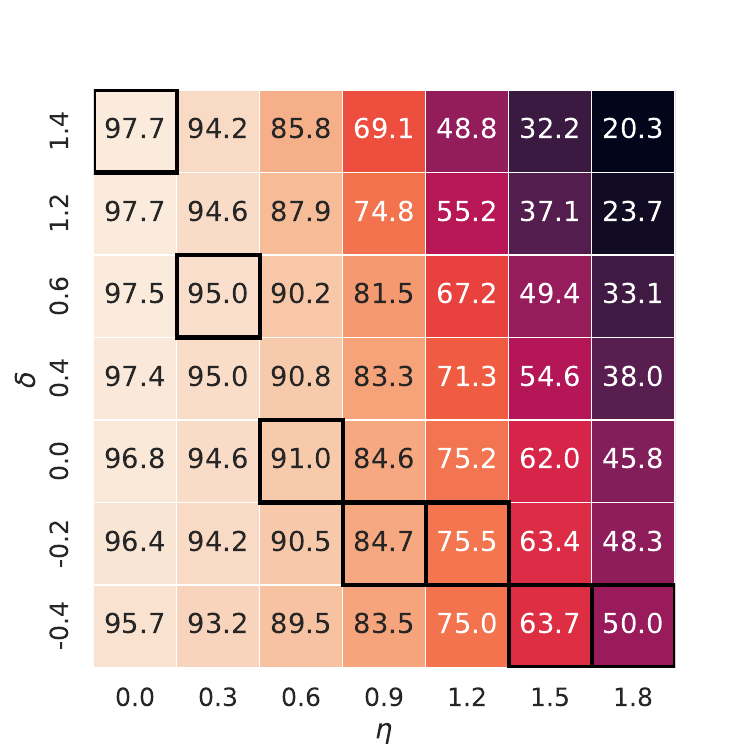}
        \caption{MNIST}
    \end{subfigure}
    \begin{subfigure}{0.425\textwidth}
        \centering
        \includegraphics[width=\textwidth]{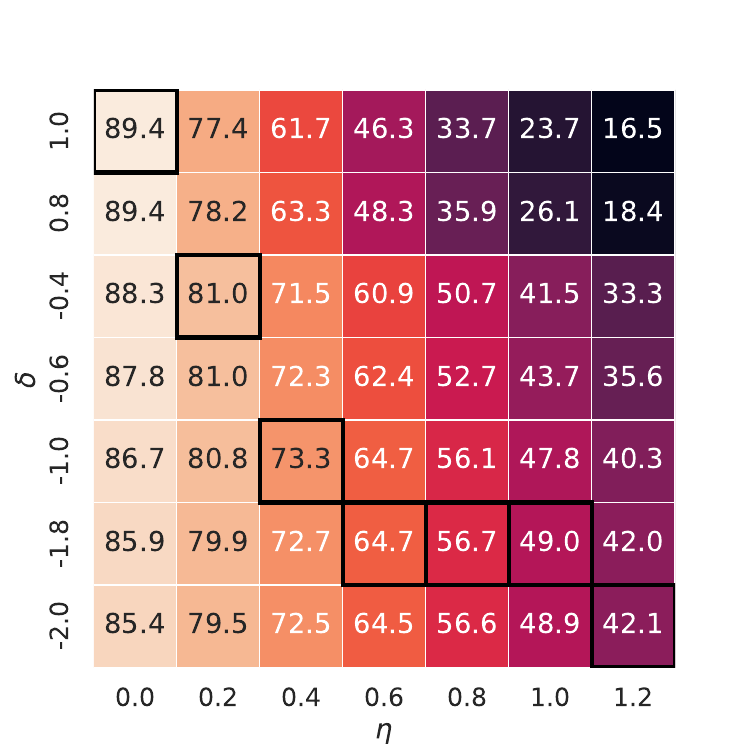}
        \caption{FashionMNIST}
    \end{subfigure}
    \\
    \begin{subfigure}{0.425\textwidth}
        \centering
        \includegraphics[width=\textwidth]{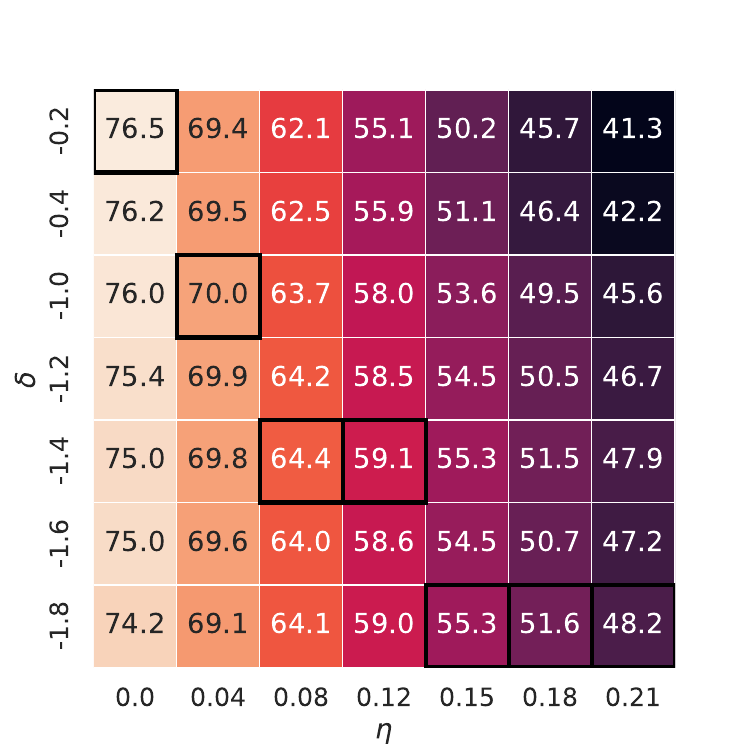}
        \caption{CIFAR10}
    \end{subfigure}
    \begin{subfigure}{0.425\textwidth}
        \centering
        \includegraphics[width=\textwidth]{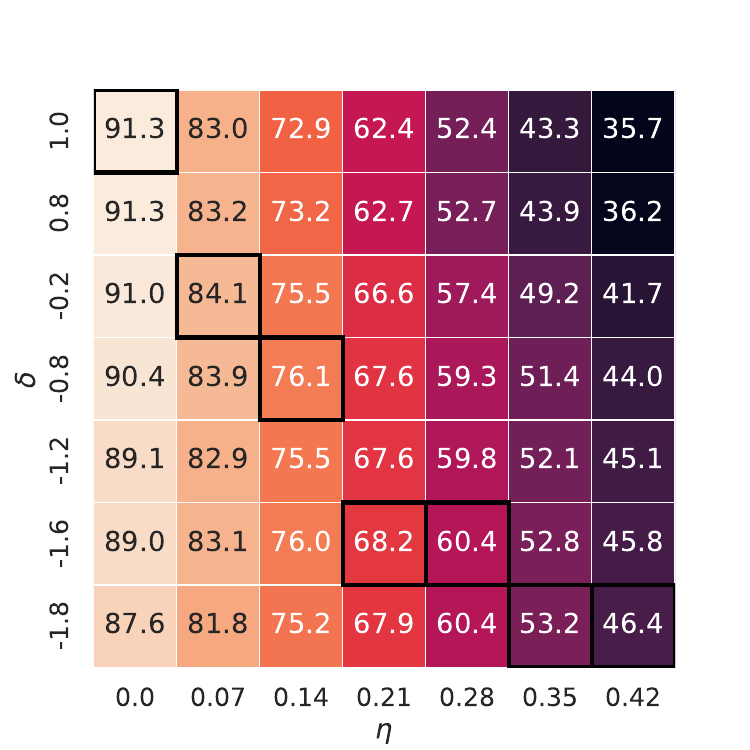}
        \caption{SVHN}
    \end{subfigure}
    \caption{Average validation accuracy of \texttt{stabODEnet} across four cross-validation folds for various $ \delta $ values, plotted as a function of FGM attack magnitude $\eta$ (measured in $\ell_2$ norm). Each row represents a different $ \delta $ setting, and the highest average validation accuracy, rounded to one decimal place, for each attack magnitude is highlighted.}
    \label{fig:2_norm}
\end{figure}

\subsection{Baselines and evaluation}\label{sec:models}

We evaluate our approach against several state-of-the-art baseline methods from the literature. Specifically, we consider the following approaches.
\begin{itemize}
    \item \textbf{\texttt{KangODEnet}}. Based on the work of \cite{kang2021stable}, this approach designs neural ODEs with equilibrium points corresponding to classification classes, ensuring Lyapunov stability so that solutions for perturbed inputs converge to the same result as unperturbed ones. The authors choose matrix $A_2$ in component (c) to minimise the maximum cosine similarity between representations from different classes, enhancing adversarial robustness. Since their specific neural ODE formulation in component (b) differs fundamentally from ours, we adopt only their $A_2$ design from component (c) to ensure fair comparison, referring to this variant as \texttt{KangODEnet}.

    \item \textbf{\texttt{nsdODEnet}}. This approach implements the structured neural ODE framework proposed by \cite{haber2017stable,ruthotto2020deep,celledoni2023dynamical,sherry2024designing}, defined as
    \[
        \dot{x}(t) = -A^\top \sigma (A x(t) + b), \quad t\in[0,1].
    \]
    This formulation is nonexpansive by design, providing inherent robustness against adversarial attacks. The prefix \texttt{nsd} denotes that the Jacobian matrix is negative semi-definite.

    \item \textbf{\texttt{shiftODEnet}}. Following \cite{guglielmi2024contractivity}, this method stabilises the weight matrix $A$ in component (b) by adding a perturbation $\Delta$ consisting of a suitable multiple of the identity matrix. We note that \cite{guglielmi2024contractivity} does not focus specifically on optimal stabilisation of weight matrix $A$, so this comparison serves primarily as an illustrative baseline to demonstrate the effectiveness of alternative stabilisation strategies.
\end{itemize}

In Tables \ref{tab:MNIST} and \ref{tab:FashionMNIST} we report the results for MNIST and FashionMNIST datasets respectively. For \texttt{stabODEnet\_d}, the perturbation matrix $E_\star(\eps_\star)\in\mathbb{S}_1\cap\mathbb{D}^{64\times 64}$ in \eqref{eq:pert} is set to be a diagonal matrix. The optimal choice of $ \delta $ for \texttt{stabODEnet\_d} is carried out as for \texttt{stabODEnet}. 
Tables \ref{tab:MNIST} and \ref{tab:FashionMNIST} show that the proposed models \texttt{stabODEnet} and \texttt{stabODEnet\_d} outperform the models \texttt{KangODEnet}, \texttt{nsdODEnet} and \texttt{shiftODEnet}. It is also interesting to notice that for the MNIST dataset the best model is \texttt{stabODEnet}, while for the FashionMNIST dataset the best model is \texttt{stabODEnet\_d}.

In Tables \ref{tab:CIFAR10} and \ref{tab:SVHN} we report the results for CIFAR10 and SVHN datasets respectively. For \texttt{stabODEnet\_d}, the perturbation matrix $E_\star(\eps_\star)\in\mathbb{S}_1\cap\mathbb{D}^{64\times 64}$ in \eqref{eq:pert} is set to be a diagonal matrix. The optimal choice of $ \delta $ for \texttt{stabODEnet\_d} is carried out as for \texttt{stabODEnet}. 
Tables \ref{tab:CIFAR10} and \ref{tab:SVHN} show that the proposed models \texttt{stabODEnet} and \texttt{stabODEnet\_d} outperform the models \texttt{KangODEnet}, \texttt{nsdODEnet} and \texttt{shiftODEnet}. It is also interesting to notice that for CIFAR10 and SVHN datasets the models \texttt{stabODEnet} and \texttt{stabODEnet\_d} behave similarly.

\begin{table}[ht]
    \centering
    \resizebox{0.85\textwidth}{!}{
    \begin{tabular}{c c c c c c c c}
        \toprule
        $\eta$ ($\ell_\infty$-norm) & 0 & 0.02 & 0.04 & 0.06 & 0.08 & 0.10 & 0.12 \\
        \midrule
        \texttt{ODEnet} & 0.9767 & 0.9205 & 0.7470 & 0.4954 & 0.2872 & 0.1426 & 0.0620 \\
        \midrule
        \texttt{stabODEnet} & 0.9772 & \textbf{0.9468} & \textbf{0.9033} & \textbf{0.8371} & \textbf{0.7408} & \textbf{0.6193} & \textbf{0.4689} \\
        \texttt{stabODEnet\_d} & \textit{0.9779} & \textit{0.9402} & \textit{0.8813} & \textit{0.7957} & \textit{0.6847} & \textit{0.5749} & \textit{0.4671} \\
        \midrule
        \texttt{KangODEnet} & \textbf{0.9791} & 0.9251 & 0.7775 & 0.5233 & 0.3048 & 0.1552 & 0.0707 \\
        \texttt{nsdODEnet} & 0.9020 & 0.8648 & 0.8108 & 0.7451 & 0.6571 & 0.5562 & 0.4372 \\
        \texttt{shiftODEnet} & 0.9101 & 0.8730 & 0.8207 & 0.7532 & 0.6687 & 0.5623 & 0.4403 \\
        \bottomrule
        \toprule
        $\eta$ ($\ell_2$-norm) & 0 & 0.3 & 0.6 & 0.9 & 1.2 & 1.5 & 1.8 \\
        \midrule
        \texttt{ODEnet} & 0.9767 & 0.9356 & 0.8283 & 0.6411 & 0.4824 & 0.3839 & 0.3281 \\
        \midrule
        \texttt{stabODEnet} & 0.9772 & \textbf{0.9504} & \textbf{0.9121} & \textbf{0.8516} & \textbf{0.7633} & \textbf{0.6405} & \textbf{0.5026} \\
        \texttt{stabODEnet\_d} & \textit{0.9779} & \textit{0.9479} & \textit{0.8985} & \textit{0.8230} & \textit{0.7253} & \textit{0.5973} & \textit{0.4763} \\
        \midrule
        \texttt{KangODEnet} & \textbf{0.9791} & 0.9392 & 0.8431 & 0.6591 & 0.4863 & 0.3736 & 0.3108 \\
        \texttt{nsdODEnet} & 0.9020 & 0.8666 & 0.8135 & 0.7490 & 0.6625 & 0.5580 & 0.4352 \\
        \texttt{shiftODEnet} & 0.9101 & 0.8746 & 0.8242 & 0.7602 & 0.6756 & 0.5693 & 0.4479 \\
        \bottomrule
    \end{tabular}}
    \caption{MNIST dataset. Comparison of test accuracy across different models as a function of the magnitude $\eta$ of the FGSM and FGM adversarial attacks. In bold and in italic the best accuracy and the second best accuracy for each perturbation size $\eta$, respectively.}
    \label{tab:MNIST}
\end{table}

\begin{table}[ht]
    \centering
    \resizebox{0.85\textwidth}{!}{
    \begin{tabular}{c c c c c c c c}
        \toprule
        $\eta$ ($\ell_\infty$-norm) & 0 & 0.01 & 0.02 & 0.03 & 0.04 & 0.05 & 0.06 \\
        \midrule
        \texttt{ODEnet} & 0.8824 & 0.7501 & 0.5807 & 0.4197 & 0.2853 & 0.1801 & 0.1094 \\
        \midrule
        \texttt{stabODEnet} & \textit{0.8844} & \textbf{0.8026} & 0.7307 & 0.6511 & 0.5727 & 0.5049 & 0.4396 \\
        \texttt{stabODEnet\_d} & 0.8835 & \textit{0.7954} & \textbf{0.7351} & \textbf{0.6832} & \textbf{0.6281} & \textbf{0.5761} & \textbf{0.5301} \\ 
        \midrule
        \texttt{KangODEnet} & \textbf{0.8846} & 0.7675 & 0.6129 & 0.4538 & 0.3163 & 0.2114 & 0.1398 \\
        \texttt{nsdODEnet} & 0.8332 & 0.7769 & 0.7089 & 0.6286 & 0.5495 & 0.4797 & 0.4143 \\
        \texttt{shiftODEnet} & 0.8404 & 0.7918 & \textit{0.7311} & \textit{0.6619} & \textit{0.5885} & \textit{0.5217} & \textit{0.4597} \\
        \bottomrule
        \toprule
        $\eta$ ($\ell_2$-norm) & 0 & 0.2 & 0.4 & 0.6 & 0.8 & 1.0 & 1.2 \\
        \midrule
        \texttt{ODEnet} & 0.8824 & 0.7548 & 0.5906 & 0.4340 & 0.3043 & 0.2060 & 0.1358 \\
        \midrule
        \texttt{stabODEnet} & \textit{0.8844} & \textbf{0.7987} & 0.7222 & 0.6373 & 0.5584 & 0.4836 & 0.4159 \\
        \texttt{stabODEnet\_d} & 0.8835 & \textit{0.7928} & \textbf{0.7315} & \textbf{0.6753} & \textbf{0.6186} & \textbf{0.5650} & \textbf{0.5154} \\ 
        \midrule
        \texttt{KangODEnet} & \textbf{0.8846} & 0.7709 & 0.6203 & 0.4655 & 0.3314 & 0.2291 & 0.1580 \\
        \texttt{nsdODEnet} & 0.8332 & 0.7732 & 0.7002 & 0.6140 & 0.5330 & 0.4594 & 0.3890 \\
        \texttt{shiftODEnet} & 0.8404 & 0.7888 & \textit{0.7238} & \textit{0.6483} & \textit{0.5729} & \textit{0.5024} & \textit{0.4359} \\
        \bottomrule
    \end{tabular}}
    \caption{FashionMNIST dataset. Comparison of test accuracy across different models as a function of the magnitude $\eta$ of the FGSM and FGM adversarial attacks. In bold and in italic the best accuracy and the second best accuracy for each perturbation size $\eta$, respectively.}
    \label{tab:FashionMNIST}
\end{table}

\begin{table}[ht]
    \centering
    \resizebox{0.85\textwidth}{!}{
    \begin{tabular}{c c c c c c c c}
        \toprule
        $\eta$ ($\ell_\infty$-norm) & 0 & 0.001 & 0.002 & 0.003 & 0.004 & 0.005 & 0.006 \\
        \midrule
        \texttt{ODEnet} & 0.7462 & 0.6663 & 0.5851 & 0.5073 & 0.4388 & 0.3754 & 0.3211 \\
        \midrule
        \texttt{stabODEnet} & \textbf{0.7623} & \textit{0.7023} & \textbf{0.6520} & \textit{0.6031} & \textit{0.5556} & \textbf{0.5154} & \textit{0.4727} \\
        \texttt{stabODEnet\_d} & \textit{0.7586} & \textbf{0.7046} & \textit{0.6500} & \textbf{0.6039} & \textbf{0.5591} & \textit{0.5136} & \textbf{0.4787} \\
        \midrule
        \texttt{KangODEnet} & 0.7406 & 0.6573 & 0.5754 & 0.4955 & 0.4256 & 0.3598 & 0.3039 \\
        \texttt{nsdODEnet} & 0.7515 & 0.6954 & 0.6356 & 0.5755 & 0.5154 & 0.4604 & 0.4103 \\
        \texttt{shiftODEnet} & 0.7513 & 0.6978 & 0.6410 & 0.5842 & 0.5287 & 0.4792 & 0.4330 \\
        \bottomrule
        \toprule
        $\eta$ ($\ell_2$-norm) & 0 & 0.04 & 0.08 & 0.12 & 0.15 & 0.18 & 0.21 \\
        \midrule
        \texttt{ODEnet} & 0.7462 & 0.6556 & 0.5646 & 0.4832 & 0.4277 & 0.3788 & 0.3339 \\
        \midrule
        \texttt{stabODEnet} & \textbf{0.7623} & \textit{0.6958} & \textbf{0.6401} & \textit{0.5853} & \textit{0.5505} & \textbf{0.5153} & \textit{0.4797} \\
        \texttt{stabODEnet\_d} & \textit{0.7586} & \textbf{0.6979} & \textit{0.6398} & \textbf{0.5888} & \textbf{0.5510} & \textit{0.5139} & \textbf{0.4845} \\
        \midrule
        \texttt{KangODEnet} & 0.7406 & 0.6462 & 0.5539 & 0.4694 & 0.4133 & 0.3586 & 0.3125 \\
        \texttt{nsdODEnet} & 0.7515 & 0.6877 & 0.6214 & 0.5546 & 0.5056 & 0.4600 & 0.4182 \\
        \texttt{shiftODEnet} & 0.7513 & 0.6905 & 0.6260 & 0.5633 & 0.5188 & 0.4784 & 0.4396 \\
        \bottomrule
    \end{tabular}}
    \caption{CIFAR10 dataset. Comparison of test accuracy across different models as a function of the magnitude $\eta$ of the FGSM and FGM adversarial attacks. In bold and in italic the best accuracy and the second best accuracy for each perturbation size $\eta$, respectively.}
    \label{tab:CIFAR10}
\end{table}

\begin{table}[ht]
    \centering
    \resizebox{0.85\textwidth}{!}{
    \begin{tabular}{c c c c c c c c}
        \toprule
        $\eta$ ($\ell_\infty$-norm) & 0 & 0.002 & 0.004 & 0.006 & 0.008 & 0.010 & 0.012 \\
        \midrule
        \texttt{ODEnet} & 0.9173 & 0.8230 & 0.6978 & 0.5719 & 0.4563 & 0.3598 & 0.2820 \\
        \midrule
        \texttt{stabODEnet} & \textbf{0.9218} & \textit{0.8609} & \textbf{0.7974} & \textbf{0.7312} & \textit{0.6585} & \textit{0.5970} & \textit{0.5325} \\
        \texttt{stabODEnet\_d} & \textit{0.9189} & \textit{0.8609} & 0.7927 & \textit{0.7246} & \textbf{0.6650} & \textbf{0.5978} & \textbf{0.5338} \\
        \midrule
        \texttt{KangODEnet} & 0.9170 & 0.8162 & 0.6857 & 0.5520 & 0.4339 & 0.3384 & 0.2627 \\
        \texttt{nsdODEnet} & 0.9071 & 0.8546 & 0.7873 & 0.7106 & 0.6312 & 0.5532 & 0.4786 \\
        \texttt{shiftODEnet} & 0.9157 & \textbf{0.8616} & \textit{0.7928} & 0.7148 & 0.6343 & 0.5557 & 0.4817 \\
        \bottomrule
        \toprule
        $\eta$ ($\ell_2$-norm) & 0 & 0.07 & 0.14 & 0.21 & 0.28 & 0.35 & 0.42 \\
        \midrule
        \texttt{ODEnet} & 0.9173 & 0.8158 & 0.6837 & 0.5525 & 0.4353 & 0.3394 & 0.2641 \\
        \midrule
        \texttt{stabODEnet} & \textbf{0.9218} & \textit{0.8551} & \textbf{0.7786} & \textbf{0.7086} & \textit{0.6298} & \textit{0.5585} & \textbf{0.4885} \\
        \texttt{stabODEnet\_d} & \textit{0.9189} & 0.8534 & 0.7782 & \textit{0.7009} & \textbf{0.6326} & \textbf{0.5587} & \textit{0.4884} \\
        \midrule
        \texttt{KangODEnet} & 0.9170 & 0.8081 & 0.6686 & 0.5293 & 0.4102 & 0.3158 & 0.2423 \\
        \texttt{nsdODEnet} & 0.9071 & 0.8486 & 0.7724 & 0.6872 & 0.6002 & 0.5165 & 0.4392 \\
        \texttt{shiftODEnet} & 0.9157 & \textbf{0.8555} & \textit{0.7785} & 0.6928 & 0.6049 & 0.5200 & 0.4453 \\
        \bottomrule
    \end{tabular}}
    \caption{SVHN dataset. Comparison of test accuracy across different models as a function of the magnitude $\eta$ of the FGSM and FGM adversarial attacks. In bold and in italic the best accuracy and the second best accuracy for each perturbation size $\eta$, respectively.}
    \label{tab:SVHN}
\end{table}

\section{Conclusions} \label{sec:conc}
We have proposed a method to enhance the stability of a neural ordinary differential equation (neural ODE) by means of a control of the maximum error growth, subsequent to a perturbation of the initial value. Since it is known that the bound depends on the logarithmic norm of the Jacobian matrix associated with the neural ODE, we have tuned this parameter by suitably perturbing the weight matrices of the neural ODE by a smallest possible perturbation (in Frobenius norm). We have done so by engaging an eigenvalue optimisation problem, for which we have proposed a nested two-level algorithm. For a given perturbation size of the weight matrix, the inner level computes optimal perturbations of that size, while - at the outer level - we tune the perturbation amplitude until we reach the desired uniform stability bound.

We have embedded the proposed algorithm in the training of the neural ODE to improve its robustness to perturbations to the initial value, which might be due simply to noisy data but also to adversarial attacks to the neural classifier. Numerical experiments on MNIST, FashionMNIST, CIFAR10 and SVHN datasets have shown that an image classifier including a neural ODE in its architecture trained according to our strategy is more stable than the same classifier trained in the classical way, and therefore it is more robust and less vulnerable to adversarial attacks. We have eventually validated our method through a range of numerical experiments against existing baseline methods to stabilise neural ODEs, demonstrating that our approach effectively improves robustness while preserving competitive accuracy, outperforming existing approaches.


\acks{N.G.\ acknowledges that his research was supported by funds from the Italian MUR (Ministero dell'Universit\`a e della Ricerca) within the PRIN 2022 Project ``Advanced numerical methods for time dependent parametric partial differential equations with applications'' and the PRO3 joint project entitled ``Calcolo scientifico per le scienze naturali, sociali e applicazioni: sviluppo metodologico e tecnologico''. N.G. and F.T. acknowledge support from MUR-PRO3 grant STANDS and PRIN-PNRR grant FIN4GEO. Part of this work was done when S.S. was a PhD student affiliated to the Division of Mathematics at Gran Sasso Science Institute, L'Aquila, Italy. The authors are members of the INdAM-GNCS (Gruppo Nazionale di Calcolo Scientifico).}



\vskip 0.2in

\end{document}